\documentclass[12pt, reqno]{amsart}
\usepackage{amsfonts,latexsym,amsthm,amssymb,amsmath,amscd,bm}
\usepackage{enumitem}
\usepackage[mathscr]{euscript}
\usepackage{longtable}

\usepackage[margin = 1in]{geometry}
\usepackage{multirow}
\usepackage{color}
\usepackage{url}
\usepackage{breakurl}
\usepackage{float}
\newcommand{\bburl}[1]{\textcolor{blue}{\url{#1}}}
\newcommand{\dfn}[1]{\textcolor{blue}{\textit{#1}}}
\usepackage{rotating}
\usepackage{tikz,graphicx}
\usepackage{circuitikz}

\usepackage{array}
\newcommand{\PreserveBackslash}[1]{\let\temp=\\#1\let\\=\temp}
\newcolumntype{C}[1]{>{\PreserveBackslash\centering}m{#1}}
\newcolumntype{R}[1]{>{\PreserveBackslash\raggedleft}p{#1}}
\newcolumntype{L}[1]{>{\PreserveBackslash\raggedright}p{#1}}

\usepackage{thmtools}
\usepackage{thm-restate}
\usepackage{hyperref}

\newtheorem{thm}{Theorem}[section]
\newtheorem{lem}[thm]{Lemma}
\newtheorem{cor}[thm]{Corollary}
\newtheorem{prop}[thm]{Proposition}
\theoremstyle{definition}
\newtheorem{defi}[thm]{Definition}
\newtheorem{algo}[thm]{Algorithm}
\newtheorem{example}[thm]{Example}

\theoremstyle{remark}

\newcommand{\R}{\mathbb{R}}
\newcommand{\C}{\mathbb{C}}

\newcommand{\Z}{\mathbb{Z}}

\usepackage[dvipsnames]{xcolor}
\newcommand{\compass}[1]{\mathbf{\color{RedOrange}
{#1}}}

\newcommand{\nc}{\newcommand}

\nc{\pre}{\text{pre}}
\nc{\id}{\text{id}}
\nc{\Inv}{\text{Inv}}
\nc{\CDyer}{\mathsf{CDyer}}
\nc{\Dyer}{\mathsf{Dyer}}

\nc{\N}{\mathbb N}
\nc{\DD}{\mathbb D}
\nc{\TT}{\mathbb T}
\nc{\EE}{\mathbb E}

\nc{\cT}{\mathcal T}
\nc{\cP}{\mathcal P}
\nc{\cM}{\mathcal M}
\nc{\cC}{\mathcal C}
\nc{\cB}{\mathcal B}
\nc{\cG}{\mathcal G}
\nc{\cD}{\mathcal D}
\nc{\cA}{\mathcal A}
\nc{\cS}{\mathcal S}
\nc{\cF}{\mathcal F}
\nc{\cL}{\mathcal L}
\nc{\cR}{\mathcal R}
\nc{\be}{\mathbf{e}}
\nc{\bv}{\mathbf{v}}
\nc{\ba}{\mathbf{a}}
\nc{\vc}{\mathbf{c}}

\numberwithin{equation}{section}
\numberwithin{part}{section}

\title
{
    Enumerating pattern-avoiding translation-invariant total orders
}

\author[Li]{Stella Jiahui Li}
\email{\textcolor{blue}{\href{mailto:stellali@college.harvard.edu}{stellali@college.harvard.edu}}}
\address{Department of Mathematics, Harvard University}

\begin{document}
\begin{abstract}
    Let $n$ be a positive integer. A translation-invariant total order (TITO) with period $n$ is a total order of the integers that is invariant under translations by multiples of $n$. These structures arise naturally in the study of Coxeter groups. In particular, real $n$-TITOs are in bijection with biclosed sets of positive roots of the affine symmetric group $\widetilde S_n$. Barkley and Defant recently introduced pattern avoidance for TITOs and used it to define the affine Tamari lattice.

    The enumeration of TITOs avoiding a single pattern of length $3$ is due to Crites and Barkley--Defant. We extend this work to TITOs avoiding two patterns. Our main results include a complete enumeration of TITOs that avoid a pair of patterns in $S_3\times S_3$, as well as of TITOs that avoid a pair $(p, q)$ with $p \in S_3 \setminus \{123, 321\}$ and $q \in S_4$. Furthermore, we provide an explicit construction of the inverse of the bijection between $312$-avoiding TITOs and noncrossing arc diagrams, thereby extending the combinatorial framework introduced by Barkley.
\end{abstract}
\maketitle

\section{Introduction}

Let $n$ be a positive integer. A \dfn{translation-invariant total order} (TITO) with period $n$, or an $n$-TITO, is a total order $\preceq$ of $\Z$ such that for all $a, b \in \Z$, we have $a \preceq b$ if and only if $a +n \preceq b + n$. We say an $n$-TITO $\preceq$ \dfn{avoids} a pattern $p \in S_k$ if there is no subsequence of integers $i_1 < \cdots < i_k$ such that $i_{p_1} \preceq i_{p_2} \preceq \cdots \preceq i_{p_k}$, where $p = p_1\cdots p_k$ in one-line notation. For instance, if $p = 312$, then we say three integers $a < b < c$ form a 312-pattern in $\preceq$ if $c \preceq a \preceq b$.\footnote{For readers who have seen definitions of pattern avoidance in other contexts, this definition may appear reversed.} Moreover, we say a TITO $\preceq$ is \dfn{$(p, q)$-avoiding} if it simultaneously avoids the patterns $p$ and $q$.

TITOs arise naturally in the study of Coxeter groups. For each Coxeter group $W$, Dyer \cite{D19} introduced the \dfn{extended weak order}, which is the containment order on biclosed sets of positive roots of $W$. The finite biclosed sets are the inversion sets of the elements of $W$, and the containment order between finite inversion sets is the weak order on $W$. Dyer conjectured that the extended weak order is always a lattice. Barkley and Speyer \cite{BS25} proved this conjecture when $W$ is an affine Coxeter group by introducing combinatorial models for biclosed sets. In particular, when $W$ is the affine symmetric group $\widetilde S_n$, the biclosed sets of positive roots are in bijection with \dfn{real} $n$-TITOs, namely those in which every block of size $1$ is waxing.

It is a classical fact that the $n$-th Tamari lattice is the sublattice of the weak order on the symmetric group $S_n$ induced by the set of $312$-avoiding permutations. Analogously, the $n$-th \dfn{affine Tamari lattice} is defined as the subposet of the extended weak order on the affine permutation group $\widetilde{S}_n$ induced by the set of real $312$-avoiding $n$-TITOs. Barkley and Defant \cite{BD25+}, who introduced the affine Tamari lattice, proved that this lattice shares many of the rich combinatorial properties of the classical Tamari lattice. Their results motivate the study of pattern avoidance for TITOs. 

The enumeration of $n$-TITOs avoiding a single pattern of length $3$ was addressed by Crites \cite{C10}, and Barkley and Defant \cite{BD25+}. In particular, for every positive integer $n > 1$, Crites showed there are infinitely many $n$-TITOs that avoid $123$ or $321$ by direct construction; for $p \in S_3 \setminus \{123, 321\}$, Barkley and Defant \cite{BD25+} showed there are $\binom{2n}{n}$ $n$-TITOs that avoid $p$ by establishing a bijection between these TITOs and noncrossing arc diagrams.

In this paper, we enumerate TITOs that simultaneously avoid two patterns $p$ and $q$. For this enumeration, we extend results from \cite{B25+} and \cite{BD25+} by giving an explicit construction of the inverse map between $312$-avoiding TITOs and noncrossing arc diagrams. Our enumerative results are summarized in Theorems \ref{thm: main thm 1} and \ref{thm: main thm 2}. We use $f_n^{(p,q)}$ to denote the number of $n$-TITOs that simultaneously avoid $p$ and $q$, and we use $\mathcal{F}^{(p, q)}(x) = \sum_{n \geq 0} f_n^{(p, q)}x^n$ to denote the ordinary generating function of $f_n^{(p, q)}$. 

As in the case of permutations, there are symmetries that preserve the number of pattern-avoiding TITOs. For instance, reversing a TITO $\preceq$ shows that the number of $n$-TITOs avoiding a pattern $p$ always equals the number of $n$-TITOs avoiding the reverse of $p$. We formalize these symmetries in $\S\ref{sec: results length 3}$. The first column of each of our tables of results lists a representative from each symmetry class.

\begin{thm}\label{thm: main thm 1}
    For distinct patterns $p, q \in S_3$, the number $f_n^{(p,q)}$ of $n$-TITOs avoiding both $p$ and $q$, together with its ordinary generating function $\mathcal{F}^{(p,q)}(x)$, is given in Table \ref{tab:two length 3 enumeration}, up to the symmetries described in $\S$\ref{sec: results length 3}.
\begin{table}[h]
    \centering
    \caption{Enumeration for distinct $p, q \in S_3$.}
    \label{tab:two length 3 enumeration}
    \begin{tabular}{ | c | c | c | } 
        \hline
         $(p, q)$& $\mathcal{F}^{(p, q)}(x)$& $f_n^{(p, q)}$  \\[0.5ex] 
         \hline 
         $(312, 123)$& $\frac{x}{1-x}$ & $1$ \\ [0.5ex] \hline
         $(312, 132)$& $\frac{2x}{1-x}$ & $2$ \\[0.5ex]  \hline
         $(312, 231), (312, 213)$ & $\frac{2x}{1-2x}$ & $2^n$ \\[0.5ex]  \hline
         $(312, 321)$ & $\frac{x}{(1-2x)(1-x)}$ & $2^n - 1$ \\ [0.5ex] \hline 
         $(123, 321)$ & $0$ & $0$ \\ [0.5ex] \hline
    \end{tabular}
\end{table}
\end{thm}

\begin{thm}\label{thm: main thm 2}
    For $q \in S_4$, the number $f_n^{(312,q)}$ of $n$-TITOs avoiding both $312$ and $q$, together with its ordinary generating function $\mathcal{F}^{(312,q)}(x)$, is given in Table \ref{tab:one length 4 enumeration intro}. Together with the symmetries described in $\S$\ref{sec: results length 3}, this determines $f_n^{(p,q)}$ for every $p \in S_3 \setminus \{123, 321\}$ and every $q \in S_4$.
    
\begin{longtable}{|C{2.5cm}|C{5cm}|C{5cm}|}
    \caption{Enumeration for $(312, q)$ with $q \in S_4$. 
    We have $\lambda_1 = \frac{1}{\alpha}, \lambda_2 = \frac{1}{\beta}, \lambda_3 = \frac{1}{\overline{\beta}}$, where 
    $\alpha, \beta, \overline{\beta}$ are roots of $3x^3 - 5x^2 + 4x - 1$.}
    \label{tab:one length 4 enumeration intro} \\
    \hline
    $(312, q)$ & $\displaystyle \mathcal{F}^{(312, q)}(x)$ & $f_n^{(312, q)}$ \\[0.5ex] 
    \hline
    \endfirsthead

    \hline
    $(312, q)$ & $\displaystyle \mathcal{F}^{(312, q)}(x)$ & $f_n^{(312, q)}$ \\[0.5ex] 
    \hline
    \endhead

    \hline
    \endfoot

    $(312, 1234)$ & \vspace{6pt}$\displaystyle \frac{x}{1-x}\vspace{6pt}$ & $1$ \\[0.5ex] \hline
    $(312, 1243)$ & \vspace{6pt}$\displaystyle \frac{2x}{1-x}\vspace{6pt}$ & $2$ \\[0.5ex] \hline
    $(312, 4321)$ & \vspace{6pt}$\displaystyle \frac{x(2x^2 - 2x + 1)}{(1-2x)(1-3x+x^2)}\vspace{6pt}$ &
       $\left(\frac{3+\sqrt{5}}{2}\right)^n + \left(\frac{3-\sqrt{5}}{2}\right)^n - 2^n $ \\[0.5ex] \hline
    $(312, 3421)$ $(312, 2431)$ & \vspace{6pt}$\displaystyle \frac{x(2 - 8x + 11x^2 - 4x^3)}{(1-x)(1-2x)(1-3x+x^2)}\vspace{6pt}$ &
       $\left(\frac{3+\sqrt{5}}{2}\right)^n + \left(\frac{3-\sqrt{5}}{2}\right)^n - 2^n + 1 $ \\[0.5ex] \hline
    $(312, 3241)$ $(312, 3214)$ $(312, 2314)$ & \vspace{6pt}$\displaystyle \frac{-x(x^2-2x+2)}{(x-1)(x^2-3x+1)}\vspace{6pt}$ &
       $\left(\frac{3+\sqrt{5}}{2}\right)^n + \left(\frac{3-\sqrt{5}}{2}\right)^n - 1 $ \\[0.5ex] \hline
    $(312, 2341)$ & \vspace{6pt}$\displaystyle \frac{x(3x^3-9x^2+6x-2)}{(1-x)(3x^3 - 5x^2 + 4x - 1)}\vspace{6pt}$ &
       $-2 + \lambda_1^n + \lambda_2^n + \lambda_3^n$ \\[0.5ex] \hline
    $(312, 2134)$ $(312, 2143)$ $(312, 1324)$ $(312, 1432)$ $(312, 1342)$ &
       \vspace{6pt}$\displaystyle \frac{2x}{1-2x}\vspace{6pt}$ & $2^n$ \\[0.5ex] 
\end{longtable}
\end{thm}

This paper is organized as follows. In $\S$\ref{sec: background}, we provide the necessary background on TITOs, pattern avoidance, affine permutations, and related combinatorial objects. In $\S$\ref{sec: arc diagrams}, we review the bijection between $312$-avoiding TITOs and noncrossing arc diagrams and give an explicit recursive construction of its inverse. In $\S$\ref{sec: results length 3}, we enumerate TITOs avoiding two patterns of length $3$. In $\S$\ref{sec: results length 4}, we enumerate TITOs that avoid both a length $3$ pattern in $S_3 \setminus \{123, 321\}$ and a pattern of length $4$. The proofs proceed by reductions to permutation enumeration in certain cases and by an analysis of noncrossing arc diagrams in others.
In the latter cases, our recursive inverse construction translates additional pattern-avoidance conditions into structural restrictions on noncrossing arc diagrams, which we then enumerate. Finally, in $\S\ref{sec: future directions}$, we discuss possible future research directions.

\section{Background}\label{sec: background}
Throughout this paper, we fix a positive integer $n$ representing the period of a TITO. We begin with the following example of a TITO when $n = 2$ to introduce structural properties of TITOs:

\begin{equation}\cdots\preceq -1 \preceq 1 \preceq 3 \preceq \cdots \preceq  4 \preceq 2 \preceq 0 \preceq\cdots
\label{eq:tito1}\end{equation}

In (\ref{eq:tito1}), observe that the TITO ``splits up'' into two components, where the odd integers all precede the even integers. In fact, this decomposition always occurs, and we can characterize the components as follows.

\begin{defi}[\cite{B25+}, Definition 4.3]
    Given any total order $\preceq$, a subset $I \subseteq \Z$ is \dfn{order-convex} if $a \preceq b \preceq c$ and $a, c \in I$ together imply that $b \in I$.

    Let $\preceq$ be a TITO. A \dfn{block} of $\preceq$ is a nonempty, order-convex subset $I$ with the following properties:
    \begin{enumerate}
        \item the ordering of $I$ by $\preceq$ has no minimal or maximal element;
        \item for any $a,c\in I$, the interval $\{b\in I\mid a\preceq b \preceq c\}$ is finite.
    \end{enumerate}
    The \dfn{size} of a block $I$ is the number of residue classes modulo $n$ appearing in $I$, which we denote by $|I|$. The sum of the sizes of all blocks of an $n$-TITO is always $n$. We say that $I$ is a \dfn{waxing} block if $x\preceq x+n$ for all $x\in I$. We say that $I$ is a \dfn{waning} block if $x+n \prec x$ for all $x\in I$. 
\end{defi}

Barkley \cite{B25+} proved that a block is either waxing or waning. The information contained in a block can be captured succinctly using a window, defined as follows.

\begin{defi}[\cite{B25+}, Section 4.1]
    A \dfn{window} for a block $I$ consists of $k$ consecutive elements $a_1, \dots, a_k$ of $I$ under $\preceq$, where $k$ is the size of $I$. If $I$ is a waxing block, we denote this window by $[a_1, \dots, a_k]$. Otherwise, $I$ must be a waning block, and we denote this window by $[\underline{a_1, \dots, a_k}].$ A \dfn{window notation} for the TITO $\preceq$ consists of a window for each of its blocks, ordered left to right in the same order that the corresponding blocks appear in $\preceq$.
\end{defi}

For example, the window notation of the TITO in (\ref{eq:tito1}) is $[1][\underline{2}].$ As shown in \cite{B25+}, each TITO has a unique block decomposition. In particular, the blocks partition $\Z$ in such a way that congruent integers are in the same block. Each block is determined by its window, and elements in the same window always belong to different congruence classes. Therefore, we specify TITOs by their window notation.

A central problem in the theory of TITOs is to determine the number of $p$-avoiding TITOs for a given pattern $p \in S_k$. Crites \cite{C10}, and Barkley and Defant \cite{BD25+} proved the following result for patterns of length 3. (We use $f_n^{(p)}$ to denote the number of $p$-avoiding TITOs with period $n$.)

\begin{thm}[\cite{C10} Theorem 1, \cite{BD25+} Corollary 5.2]
    For $p \in S_3$ and $n\geq 2$, we have
    \begin{align*}
        f_n^{(p)} = \begin{cases}
            \infty & \text{if } p \in \{123, 321\};  \\
            \binom{2n}{n} &\text{otherwise.}
        \end{cases}
    \end{align*}
\end{thm}

In this paper, our goal is to understand and enumerate TITOs that simultaneously avoid pairs of patterns. Specifically, we consider:
\begin{enumerate}
    \item pairs $(p, q) \in S_3 \times S_3$ with $p \neq q$;
    \item pairs $(p, q) \in \left(S_3 \setminus \{123, 321\}\right) \times S_4$.
\end{enumerate}

The following lemma classifies the block structure of $312$-avoiding TITOs, which will be crucial to our subsequent enumeration.  

\begin{lem}[\cite{BD25+}, Lemmas $3.7$ and $3.8$]\label{lem:312 blocks}
    Let $\preceq$ be a $312$-avoiding TITO. Then $\preceq$ has at most two blocks. If $\preceq$ has two blocks, then the left block is waxing and the right block is waning. If $\preceq$ has a waning block, then the elements of the waning block appear in decreasing order.
\end{lem}

We also make the following observation about the windows in the window notation of $312$-avoiding TITOs.

\begin{lem}\label{lem:windows of 312}
    Let $\preceq$ be a $312$-avoiding TITO. Fix a window notation, and let $a, b$ be any two integers in the same window with $a < b$ and $b \preceq a$. Then $b - a < n$.
\end{lem}
\begin{proof}
    Assume for contradiction that $b - a \geq n$. Since $b \not\equiv a \pmod{n}$, we know $b - a > n$. Recall that the elements of a window are consecutive in $\preceq$ among the elements of their block: if the block is waxing with window $[a_1, \dots, a_k]$, then $a_1 \preceq \cdots \preceq a_k \preceq a_1 + n$, and if it is waning, then $a_1 \preceq \cdots \preceq a_k \preceq a_1 - n$. In particular, $a \preceq b + n$ in the waxing case and $a \preceq b - n$ in the waning case.

    If $a$ and $b$ are in the same waxing block, then $$b \preceq a \preceq b+n \preceq a+n,$$ and $b, a, a+n$ form a $312$ pattern. 

    If $a$ and $b$ are in the same waning block, then $$b+n \preceq b \preceq a \preceq b-n,$$ and $b+n, a, b-n$ form a $312$ pattern. 
\end{proof}

We introduce the following helpful notation. Affine permutations and projective TITOs are defined and discussed in more detail in $\S \ref{sec: affine permutations}$ and $\S \ref{sec: titos count}$, respectively. Here, a TITO is called \dfn{projective} if it is not an affine permutation.
\begin{align*}
    f_n^{(p, q)} &= \text{number of $(p, q)$-avoiding TITOs}; \\
    s_n^{(p, q)} &= \text{number of $(p, q)$-avoiding permutations in $S_n$}; \\
    a_n^{(p, q)} &= \text{number of $(p, q)$-avoiding affine permutations in $\tilde{S}_n$}; \\
    b_n^{(p, q)} &= \text{number of $(p, q)$-avoiding projective TITOs} =  f_n^{(p, q)} - a_n^{(p, q)}.
\end{align*}

For each of the quantities above, we define the corresponding ordinary generating function:
\begin{align*}
    &\mathcal{F}^{(p, q)}(x) = \sum_{n \geq 0} f_n^{(p, q)}x^n; \\ &\mathcal{S}^{(p, q)}(x) = \sum_{n \geq 0} s_n^{(p, q)}x^n; \\
    &\mathcal{A}^{(p, q)}(x) = \sum_{n \geq 0} a_n^{(p, q)}x^n; \\ &\mathcal{B}^{(p, q)}(x) = \sum_{n \geq 0} b_n^{(p, q)}x^n.
\end{align*}

Throughout, we adopt the conventions $f_0^{(p,q)} = a_0^{(p,q)} = b_0^{(p,q)} = 0$, since there is no TITO or affine permutation of period $0$, and $s_0^{(p,q)} = 1$, counting the empty permutation. In particular, $\mathcal{F}^{(p,q)}$, $\mathcal{A}^{(p,q)}$, and $\mathcal{B}^{(p,q)}$ have constant term $0$, while $\mathcal{S}^{(p,q)}$ has constant term $1$.

\subsection{Permutations}
The following lemma gives the enumeration for $(p, q)$-avoiding permutations in $S_n$ for $(p, q) \in S_3 \times S_3$ and $p \neq q$.

\begin{lem}[\cite{SS85}, Propositions $7$--$11$]\label{lem:pattern avoiding perm}
    For every $n \geq 1$, we have
        $$s_n^{(p,q)} = \begin{cases}
            2^{n-1} & \text{if }\{p, q\} \in \{\{312, 321\}, \{312, 231\}, \{312, 213\}, \{312, 132\}, \{123, 132\},\\
            & \qquad \quad \ \{123, 213\}, \{231, 321\}, \{132, 213\}, \{132, 231\}, \{213, 231\}\}; \\
            \binom{n}{2} + 1 & \text{if }\{p, q\} \in \{\{312, 123\}, \{132, 321\}, \{123, 231\}, \{213, 321\}\}; \\
            0 & \text{if }\{p, q\} = \{123, 321\} \text{ and } n \geq 5; \\
            1, 2, 4, 4 & \text{if } \{p, q\} = \{123, 321\} \text{ and }
              n = 1, 2, 3, 4, \text{ respectively.}
        \end{cases}$$
\end{lem}

Lemma \ref{lem:pattern avoiding perm} allows us to determine $\mathcal{S}^{(p, q)}(x)$ for all $(p, q) \in S_3 \times S_3$ and $p \neq q$. This will be useful later when we compute the generating functions $\mathcal{F}^{(p,q)}(x)$.

\begin{lem}\label{lem: generating function perm}
    For distinct $p,q \in S_3$, we have
    \begin{equation*}
        \mathcal{S}^{(p,q)}(x) = \begin{cases}
             \frac{1-x}{1-2x} &\text{if }s_{n}^{(p,q)} = 2^{n-1} \text{ for $n \geq 1$}; \\
            \frac{1-2x+2x^2}{(1-x)^3} &\text{if }s_{n}^{(p,q)} = \binom{n}{2} + 1\text{ for $n \geq 1$};\\
            1 + x + 2x^2 + 4x^3 + 4x^4 &\text{if } \{p,q\} = \{123, 321\}.
        \end{cases}
    \end{equation*}
\end{lem}

It is also possible to compute $\mathcal{S}^{(231,p)}(x)$ for a length-4 pattern $p$ that itself avoids $231$. We include West's formulas \cite{W96} below for use in our calculations later.

\begin{lem}[\cite{W96}, Table 1]\label{lem: perm generating functions}
    For $p \in \{4321, 4123\}$, we have the following generating functions $\mathcal{S}^{(231, 4321)}(x)$ and $\mathcal{S}^{(231, 4123)}(x)$:
    \begin{align*}
        \mathcal{S}^{(231, 4321)}(x) &= \sum_{n \geq 0} s_n^{(231, 4321)} x^n = \frac{1-2x}{1-3x+x^2}, \\
        \mathcal{S}^{(231, 4123)}(x) &= \sum_{n \geq 0} s_n^{(231, 4123)} x^n = \frac{(1-x)^3}{1-4x+5x^2-3x^3}.
    \end{align*}
\end{lem}

It is also helpful to know the number of permutations that avoid patterns $p = p_1\cdots p_k$ of length $k$ such that $p_1 = k$. Given $p \in S_n$, we use $p_{[i:j]}$ to denote $p_i p_{i+1} \cdots p_j$, where $1\leq i \leq j \leq n$.\footnote{Whenever a subword is used as a forbidden pattern, we mean its standardization, obtained by replacing its entries by their ranks. For example, if $p=1432$, then avoiding $p_{[2:4]}=432$ means avoiding $321$.}

\begin{lem}\label{lem: count permutations n---}
    Let $p \in S_k$ such that $p_1 = k$. The number of $(231, p)$-avoiding permutations in $S_n$ is given by 
    $$s_n^{(231, p)} = \sum_{i = 1}^n s_{i-1}^{(231, p)}s_{n-i}^{(231, p_{[2:k]})}.$$
\end{lem}
\begin{proof}
    Let $u$ be a $(231, p)$-avoiding permutation in $S_n$, and let $i$ be the index such that $u_i = n$. Since $u$ avoids $231$, every entry to the left of position $i$ is smaller than every entry to its right: if $a < i < b$ with $u_a > u_b$, then $u_a, n, u_b$ would be an occurrence of $231$. Hence $u_{[1: i-1]}$ is a permutation of $\{1, \dots, i-1\}$ and $u_{[i+1: n]}$ is (as a pattern) a permutation in $S_{n-i}$. Consequently, any occurrence of $p$ in $u$ that does not use position $i$ lies entirely on one side of position $i$, since its first entry is its largest ($p_1 = k$) while entries on the left are smaller than those on the right. Moreover, any occurrence of $p$ that uses position $i$ must use it as its first entry, so the remaining entries form an occurrence of $p_{[2:k]}$ to the right of position $i$. Therefore $u$ avoids $(231, p)$ if and only if $u_{[1: i-1]}$ avoids $(231, p)$ and $u_{[i+1: n]}$ avoids $(231, p_{[2:k]})$. Summing over $i$ gives the formula.
\end{proof}

The following statement is an immediate consequence of Lemma \ref{lem: count permutations n---}.

\begin{lem}\label{lem: same count perm}
    Let $p, p' \in S_k$ be such that $p_1 = p'_1 = k$. If $s_m^{(231, p_{[2:k]})} = s_m^{(231, p'_{[2:k]})}$ for all $m < n$, then $s_n^{(231, p)} = s_n^{(231, p')}$.
\end{lem}

\subsection{Affine permutations}\label{sec: affine permutations}

Affine permutations are a generalization of permutations. \allowbreak Barkley and Defant \cite{BD25+} showed that affine permutations are in bijection with TITOs consisting of exactly one waxing block. By Lemma \ref{lem:312 blocks}, any $312$-avoiding TITO has one of three possible block structures; one such structure consists of exactly one waxing block. The $(312, p)$-avoiding TITOs with this block structure can therefore be enumerated by counting $(312, p)$-avoiding affine permutations. This subsection provides the necessary definitions and lemmas to support our calculation.

\begin{defi}
    An \dfn{affine permutation} of size $n$ is a bijection $\omega: \Z \to \Z$ such that 
    \begin{enumerate}
        \item $\omega(i+n) = \omega(i) + n$ for all $i \in \Z$, and 
        \item $\sum_{i = 1}^n \omega(i) = \binom{n+1}{2}$.
    \end{enumerate}

    The \dfn{affine symmetric group} $\tilde{S}_n$ is the group of all affine permutations of a fixed size $n$. We abbreviate $\omega(i)$ as $\omega_i$. 
\end{defi}

We can think of $[\omega_1, \dots, \omega_n]$ as the \dfn{window notation} of an affine permutation $\omega$, and we refer to it as the \dfn{base window} of $\omega$. We now introduce ideas from \cite{C10} that will help count pattern-avoiding affine permutations in $\tilde{S}_n$ by transforming their base windows into permutations in $S_n$.

We define a map $\sigma_r : \tilde{S}_n \to \tilde{S}_n$ by setting 
\begin{align*}
    \sigma_r(\omega)_i = \begin{cases}
        \omega_{i-1} + 1 & \text{if } 2 \leq i \leq n, \\
        \omega_n - n + 1, &\text{if }i = 1.
    \end{cases}
\end{align*}

The map $\sigma_r$ acts on the window notation of $\omega$ by cyclically shifting each entry one position to the right: the entry occupying position $i$ moves to position $i+1$ (indices taken mod $n$), so the last entry of the window wraps around into the first position. To keep the result a valid affine permutation, this wrapped entry is also decreased by $n$, and every entry of the window (including the wrapped one) is increased by $1$. Concretely, if $\omega$ has window $[\omega_1, \dots, \omega_n]$, then $\sigma_r(\omega)$ has window $[\omega_n - n + 1,\ \omega_1 + 1,\ \dots,\ \omega_{n-1}+1]$.\footnote{The map $\sigma_r$ is an automorphism of $\tilde{S}_n$ of order $n$ obtained by rotating the Coxeter graph of $\tilde{S}_n$ one space clockwise.} Set $\sigma_l = \sigma_r^{-1}$, which has the effect of shifting the window one position to the left, with the analogous adjustment applied to the entry that wraps from the first to the last position.

\begin{example}
    If $\omega = [5, -4, 6, 3]$, then $\sigma_r(\omega) = [0, 6, -3,7]$ and $\sigma_l(\omega) = [-5, 5, 2, 8]$. 
\end{example}

As in the case of TITOs, we say an affine permutation $\omega \in \tilde{S}_n$ \dfn{avoids} the pattern $p \in S_k$ if there is no subsequence of integers $i_1 < \cdots < i_k$ such that the subword $\omega_{i_1}\cdots\omega_{i_k}$ has the same relative order as the elements of $p$. The following lemma shows that $\sigma_r$ and $\sigma_l$ preserve relative ordering and will help us enumerate pattern-avoiding affine permutations.

\begin{lem}[\cite{C10}, Lemmas $8$--$9$]\label{lem:inversion}
    Let $\omega \in \tilde{S}_n$ and $p \in S_m$. The following are equivalent: 
    \begin{enumerate}
        \item $\omega$ avoids $p$;
        \item $\omega^{-1}$ avoids $p^{-1}$;
        \item $\sigma_r(\omega)$ avoids $p$;
        \item $\sigma_l(\omega)$ avoids $p$.
    \end{enumerate}
\end{lem}

The lemma below allows us to shift the base window of any $231$-avoiding affine permutation in $\tilde{S}_n$ to a $231$-avoiding permutation in $S_n$ (i.e., all elements in the base window are in $[n]$), which is very helpful for enumeration. Note that $S_n \subset \tilde{S}_n$ by identifying a permutation with the affine permutation having it as base window.

\begin{lem}[\cite{C10}, Proof of Theorem 7]\label{lem:231 avoiding in S_n}
    Let $\omega \in \tilde{S}_n$ be $231$-avoiding, and let $\alpha$ be the index such that $$\omega_{\alpha} = \max\{\omega_1, \dots, \omega_n\}.$$ Then $u = \sigma_l^{\omega_\alpha - n}(\omega) \in S_n \subset \tilde{S}_n$ and $u_{\alpha - \omega_{\alpha} + n} = n$.
\end{lem}

Let $\beta = \alpha - \omega_{\alpha} + n$. Note that the possible values of $\beta$ are $1$ through $\alpha$, and the possible values of $\alpha$ are $1$ through $n$. 

\begin{example}
    Let $\omega = [0,2,3,5]$. Then $\alpha = 4$, $\omega_{\alpha} = 5$, and $n = 4$. Thus, $u = \sigma_l(\omega) = [1, 2, 4, 3]$, and the index of $4$ in $u$ is $u^{-1}(n) = \alpha - \omega_{\alpha} + n = 4 - 5 + 4 = 3$.
\end{example}

Now, we introduce some lemmas that classify affine permutations that simultaneously avoid $231$ and another pattern $p$, which will be helpful for our enumeration. 

\begin{lem}\label{lem:avoiding for n---}
    Let $p \in S_k$ with $p_1 = k$. An affine permutation with window $u \in S_n$ such that $u_i = n$ is $(231, p)$-avoiding if and only if $u_{[1: i-1]} \in S_{i-1}$ is $(231, p)$-avoiding and $u_{[i+1: n]}$ is $(231, p_{[2: k]})$-avoiding. 
\end{lem}
\begin{proof}   
    First, note that an affine permutation with base window represented by $u \in S_n$ is $p$-avoiding if and only if the permutation $u$ is $p$-avoiding. The problem therefore reduces to proving the statement for the permutation $u$, which is addressed in Lemma \ref{lem: count permutations n---}.
\end{proof}

We now characterize the affine permutations avoiding $231$ and patterns $p \in S_k$ with $p_1 = 1$ and $p_2 = k$.

\begin{lem}\label{lem:avoiding for 1n--}
    Let $p \in S_k$ with $p_1 = 1$ and $p_2 = k$. An affine permutation with window $u \in S_n$ such that $u_i = n$ is $(231, p)$-avoiding if and only if $u_{[1: i-1]} \in S_{i-1}$ is $(231, p_{[2:k]})$-avoiding and $u_{[i+1: n]}$ is $(231, p_{[3:k]})$-avoiding. 
\end{lem}
\begin{proof}    
    Since $u \in S_n \subset \tilde{S}_n$, we know $u_a - n = u_{a-n} < u_b < u_{a+n} = u_a + n$ for all $1 \leq a, b \leq n$. Therefore, the affine permutation is $p$-avoiding if and only if $u$ is $p_{[2:k]}$-avoiding. By Lemma \ref{lem: count permutations n---}, we know $u$ is $(231, p_{[2:k]})$-avoiding if and only if $u_{[1: i-1]} \in S_{i-1}$ is $(231, p_{[2:k]})$-avoiding and $u_{[i+1: n]}$ is $(231, p_{[3:k]})$-avoiding. 
\end{proof}

Analogously, we can characterize affine permutations avoiding $231$ and patterns $p = p_1 \cdots p_k$ with $p_1 = k-1$ and $p_k = k$. The proof is similar to that of Lemma \ref{lem:avoiding for 1n--} and is therefore omitted.
\begin{lem}\label{lem:avoiding for ---n}
    Let $p \in S_k$ with $p_1 = k-1$ and $p_k = k$. An affine permutation with window $u \in S_n$ such that $u_i = n$ is $(231, p)$-avoiding if and only if $u_{[1: i-1]} \in S_{i-1}$ is $(231, p_{[1:k-1]})$-avoiding and $u_{[i+1: n]}$ is $(231, p_{[2:k-1]})$-avoiding. 
\end{lem}
The hypothesis $p_1 = k-1$ guarantees that every occurrence of $p_{[1:k-1]}$ lies within a single window, since $p_{[1:k-1]}$ begins with its largest entry.

Combining Lemmas \ref{lem:231 avoiding in S_n} and \ref{lem:avoiding for n---} allows us to enumerate the affine permutations avoiding $(231, p)$, where the pattern $p \in S_k$ satisfies $p_1 = k$. 

\begin{lem}\label{lem: formula affine avoid n---}
    For $p \in S_k$ with $p_1 = k$, the number of $(231, p)$-avoiding affine permutations is $$a_n^{(231, p)} = \sum_{\alpha = 1}^n \left(\sum_{\beta = 0}^{\alpha - 1} s_{\beta}^{(231, p)}s_{n-\beta-1}^{(231, p_{[2:k]})}\right) = \sum_{j = 0}^{n-1} (n-j) s_{j}^{(231, p)}s_{n-j-1}^{(231, p_{[2:k]})}.$$
\end{lem}
\begin{proof}
    Let $\omega$ be a $(231, p)$-avoiding affine permutation. By Lemma \ref{lem:231 avoiding in S_n}, $\sigma_l^{\omega_{\alpha} - n}(\omega) = u  \in S_n$ with $u_{\beta} = n$ for some $1 \leq \beta \leq \alpha$. Since $\sigma_l$ is an automorphism that preserves relative order, we know $u$ is also $(231, p)$-avoiding. Thus, it suffices to count the permutations $u \in S_n$ avoiding $(231, p)$. By Lemma $\ref{lem:avoiding for n---}$, there are $s_{\beta-1}^{(231, p)}s_{n-\beta}^{(231, p_{[2:k]})}$ such permutations.

    Conversely, fix $1\leq\beta\leq\alpha\leq n$ and any such $u$ with $u_\beta=n$. Set $\omega=\sigma_r^{\alpha-\beta}(u)$, which avoids $(231,p)$ by Lemmas~\ref{lem:inversion} and~\ref{lem:avoiding for n---}. Since $\alpha-\beta\leq n-\beta$, the entry $n$ does not wrap, and the unique maximum of the resulting window is $\omega_\alpha=n+\alpha-\beta$. Thus, $\alpha$ is recovered from $\omega$, and $u=\sigma_l^{\omega_\alpha-n}(\omega)$, which establishes the required bijection.
    
    Summing over all possible values of $\beta$ tells us that the number of $(231, p)$-avoiding affine permutations is
    \begin{align*}
        \sum_{\beta = 1}^{\alpha} s_{\beta-1}^{(231, p)}s_{n-\beta}^{(231, p_{[2:k]})} = \sum_{\beta = 0}^{\alpha - 1} s_{\beta}^{(231, p)}s_{n-\beta - 1}^{(231, p_{[2:k]})}.
    \end{align*}

    The index $\alpha$ ranges from $1$ to $n$, so summing over all $1 \leq \alpha \leq n$ gives
    \begin{align*}
        a_n^{(231, p)} = \sum_{\alpha = 1}^n \left(\sum_{\beta = 0}^{\alpha - 1} s_{\beta}^{(231, p)}s_{n-\beta-1}^{(231, p_{[2:k]})}\right) = \sum_{j = 0}^{n-1} (n-j)s_j^{(231, p)}s_{n-j-1}^{(231, p_{[2:k]})},
    \end{align*}
    as desired. 
\end{proof}

Combining Lemmas \ref{lem:231 avoiding in S_n} and \ref{lem:avoiding for 1n--} also allows us to enumerate the affine permutations avoiding $(231, p)$, where $p \in S_k$ is a pattern with $p_1 = 1$ and $p_2 = k$. We omit the proof as it is similar to the proof above.
\begin{lem}\label{lem: formula affine avoid 1n--}
    For $p \in S_k$ with $p_1 = 1$ and $p_2 = k$, the number of $(231, p)$-avoiding affine permutations is $$a_n^{(231, p)} = \sum_{\alpha = 1}^n \left(\sum_{\beta = 0}^{\alpha - 1} s_{\beta}^{(231, p_{[2:k]})}s_{n-\beta-1}^{(231, p_{[3:k]})}\right) = \sum_{j = 0}^{n-1} (n-j) s_{j}^{(231, p_{[2:k]})}s_{n-j-1}^{(231, p_{[3:k]})}.$$
\end{lem}

\begin{lem}\label{lem: formula affine avoid ---n}
    For $p \in S_k$ with $p_1 = k-1$ and $p_k = k$, the number of $(231, p)$-avoiding affine permutations is $$a_n^{(231, p)} = \sum_{\alpha = 1}^n \left(\sum_{\beta = 0}^{\alpha - 1} s_{\beta}^{(231, p_{[1:k-1]})}s_{n-\beta-1}^{(231, p_{[2:k-1]})}\right) = \sum_{j = 0}^{n-1} (n-j) s_{j}^{(231, p_{[1:k-1]})}s_{n-j-1}^{(231, p_{[2:k-1]})}.$$
\end{lem}

Finally, from Lemmas \ref{lem: same count perm} and \ref{lem: formula affine avoid n---}, we obtain the following.

\begin{lem}\label{lem: same count affine}
    Let $p, p' \in S_k$ such that $p_1 = p'_1 = k$. If $s_n^{(231, p_{[2:k]})} = s_n^{(231, p'_{[2:k]})}$ for all $n$, then $a_n^{(231, p)} = a_n^{(231, p')}$ for all $n$. 
\end{lem}
\section{Noncrossing arc diagrams}\label{sec: arc diagrams}

There is a bijection between $312$-avoiding TITOs and combinatorial objects called noncrossing arc diagrams. Barkley and Defant \cite{BD25+} used these diagrams to enumerate $312$-avoiding TITOs. We extend results from \cite{BD25+} by giving an explicit recursive construction of window notation for the inverse of the bijection between $312$-avoiding TITOs and noncrossing arc diagrams in $\S\ref{subsec:inverse construction}$. This construction will help us enumerate $(312, p)$-avoiding TITOs in $\S\ref{subsubsec: arc diagram approach}$ later.

We first discuss the forward map from $312$-avoiding TITOs to noncrossing arc diagrams. 

\begin{defi}
    A \dfn{reflection index} is a pair $(a, b)$ of integers with $a < b$, considered up to simultaneous translations by multiples of $n$ on each coordinate. We always consider reflection indices with $a \in [n]$. An \dfn{inversion} of a TITO $\preceq$ is a reflection index $(a, b)$ such that $a \succeq b$. 

    For $a, b \in \Z$ with $a \preceq b$, the \dfn{interval} between $a$ and $b$ is the set $[a, b] = \{c \in \Z \mid a \preceq c \preceq b\}$. If $[a, b]$ has cardinality $2$, then we say $b$ \dfn{covers} $a$ and $a \preceq b$ is a \dfn{cover relation}. A \dfn{wall} of a TITO $\preceq$ is a reflection index $(a, b)$ such that $a \preceq b$ is a cover relation or $b \preceq a$ is a cover relation. We say a wall $(a, b)$ is an \dfn{upper wall} if $a \preceq b$; otherwise, we say it is a \dfn{lower wall}.
\end{defi}

\begin{example}
    The set of lower walls for the TITO $[1][\underline{2}]$ (\ref{eq:tito1}) is $\{(2,4)\}$.
\end{example}

\begin{defi}
    Consider an annulus whose outer boundary has $n$ marked points $v_1,\ldots,v_n$ in clockwise order. We regard the marked points as elements of $\Z/n\Z$. An \dfn{arc} is a simple curve inside the annulus directed clockwise around the puncture that starts at one marked boundary point and ends at another marked boundary point (potentially the same as its starting point). We denote an arc from $v_i$ to $v_j$ by $\gamma_{i, j}$.
    
    Arcs are considered up to isotopy, so they are in one-to-one correspondence with the reflection indices $(a,b)$ such that $b\leq a+n$. The arc corresponding to a reflection index $(a,b)$ is $\gamma_{a,b}$ if $b \leq n$ and $\gamma_{a, b-n}$ if $a + n \geq b > n$. A \dfn{noncrossing arc diagram} is a collection of arcs that can be drawn so that they do not intersect in their interiors and such that each marked point has at most one incoming arc and at most one outgoing arc. (The second condition is not implied by the first because two arcs that share an initial point or share a terminal point have disjoint interiors, yet such a pair is not allowed.)
    
    Note that each noncrossing arc diagram has at most one cycle.
\end{defi}

\begin{defi}
    Let $\preceq$ be a $312$-avoiding TITO with period $n$. The \dfn{arc diagram} of $\preceq$, denoted $\cA(\preceq) = (V, E)$, is the set of boundary points $V = \{v_1, \dots, v_n\}$ and the set of arcs
    $$ E = \{\gamma_{a,b}: (a, b) \text{ or } (a, b+n) \text{ is a lower wall for }a, b \in [n]\}. $$
    (We check both $(a,b)$ and $(a,b+n)$ because $b \in [n]$ only represents its residue class mod $n$, and the integer actually covering or covered by $a$ could correspond to either translate.)
\end{defi}
Note that, by construction, each boundary point in $\cA(\preceq)$ has at most one incoming arc and one outgoing arc, since each $a \in [n]$ is covered in $\preceq$ by exactly one element and covers exactly one element. Also, since reflection indices are considered up to simultaneous translations by multiples of $n$, we can determine the arc diagram of a TITO directly from its window notation without first reconstructing the corresponding TITO.

Barkley and Defant \cite{BD25+} exhibited a bijection between noncrossing arc diagrams and $312$-avoiding TITOs to enumerate $312$-avoiding TITOs. We restate their theorem below, and we give an explicit inverse construction to enumerate $(312, p)$-avoiding TITOs in $\S\ref{subsubsec: arc diagram approach}$.

\begin{thm}[\cite{BD25+}, Proposition 5.1]\label{thm:arcdiagram312}
    The map $\cA$ is a bijection from the set of $312$-avoiding $n$-TITOs to the set of noncrossing arc diagrams with $n$ marked boundary points. In particular, a $312$-avoiding TITO is determined by its collection of lower walls.
\end{thm}

\begin{example}
    Let $n = 7$. For the 312-avoiding TITO $[1, 4, 3, 6, 5][\underline{7, 2}]$, the lower walls are $\{(2, 7), (3, 4), (5, 6), (7, 9)\}$, which correspond to arcs $\gamma_{2, 7}, \gamma_{3, 4}, \gamma_{5,6},$ and $\gamma_{7,2}$, respectively. We obtain the noncrossing arc diagram in Figure~\ref{fig:arcdiagram1}.
\begin{figure}[H]
    \centering    
\scalebox{0.7}{\begin{tikzpicture}[
    dot/.style={circle, fill, inner sep=1.5pt},
    arrow/.style={->, blue, thick}]
    \def\radius{3cm}

    \draw[thick, fill=gray!15] (0,0) circle (\radius);

    \draw[thick, fill=white, draw=black] (0,0) circle (5pt);

    \foreach \p/\pos in {1/above right, 2/right, 3/below right, 4/below, 5/below left, 6/left, 7/above} {
        \pgfmathsetmacro{\angle}{90 + (7-\p)*360/7}
        
        \ifnum\p=1
            \node[dot, label=\pos:$v_{\p}$] (\p) at (\angle:\radius) {};
        \else
            \node[dot, label=\pos:$v_{\p}$] (\p) at (\angle:\radius) {};
        \fi
    }

    \draw[arrow, looseness=2] (2) to [out=-170, in=-110] (7);

    \draw[arrow, looseness=0.1] (7) to[out=-30] (2);

    \draw[arrow, looseness=1] (5) to[out=60, in=-90] (6);

    \draw[arrow, looseness=1] (3) to[out=135, in=45] (4);
\end{tikzpicture}
}
\caption{The noncrossing arc diagram of the $312$-avoiding TITO $[1, 4, 3, 6, 5][\underline{7, 2}]$.}
    \label{fig:arcdiagram1}
\end{figure}
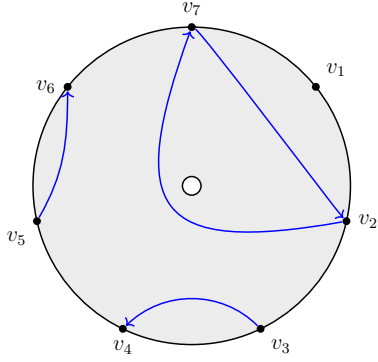
\end{example}
    Note that the puncture allows us to distinguish $\gamma_{a, b}$ from $\gamma_{b, a}$ even when we consider arcs up to isotopy.

\begin{example}    
    The $312$-avoiding TITO $[1,6,5,7,4,3][\underline{2}]$ has lower walls $\{(5,6),\allowbreak (4,7),\allowbreak (3,4),\allowbreak (2,9)\}$, which correspond to arcs $\gamma_{5,6}, \gamma_{4,7}, \gamma_{3,4}$ and $\gamma_{2,2}$, respectively.  We obtain the noncrossing arc diagram in Figure \ref{fig:arcdiagram2}.
    \begin{figure}[H]
    \scalebox{0.7}{
    \begin{tikzpicture}[
    dot/.style={circle, fill, inner sep=1.5pt},
    arrow/.style={->, blue, thick}
]
    \def\radius{3cm}

    \draw[thick, fill=gray!15] (0,0) circle (\radius);

    \draw[thick, fill=white, draw=black] (0,0) circle (5pt);

    \foreach \p/\pos in {1/above right, 2/right, 3/below right, 4/below, 5/below left, 6/left, 7/above} {
        \pgfmathsetmacro{\angle}{90 + (7-\p)*360/7}
        
        \ifnum\p=1
            \node[dot, label=\pos:$v_\p$] (\p) at (\angle:\radius) {};
        \else
            \node[dot, label=\pos:$v_\p$] (\p) at (\angle:\radius) {};
        \fi
    }

    \draw[arrow, looseness=1.0] (4) to[out=100, in=-120] (7);

    \draw[arrow, looseness=1] (5) to[out=60, in=-90] (6);

    \draw[arrow, looseness=1.0] (3) to[out=135, in=45] (4);
    
    \coordinate (loop_via) at (-0.5, 0);
    
    \draw[arrow, looseness=1.0] (2) to[out=200, in=-90] (loop_via) to[out=90, in=110] (2);

\end{tikzpicture}
    }
\caption{The noncrossing arc diagram of the $312$-avoiding TITO $[1, 6, 5, 7, 4, 3][\underline{2}]$.}
    \label{fig:arcdiagram2}
\end{figure}
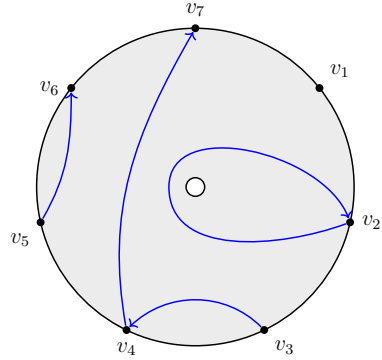
\end{example}

\begin{defi}
    We view arcs as \dfn{directed edges} between boundary points. Denote the \dfn{length} of an arc $\gamma_{a, b}$ by $\ell(\gamma_{a,b})$, and set
    \begin{align*}
        \ell(\gamma_{a,b}) = \begin{cases}
            b - a &\text{if } b > a; \\
            b+n - a &\text{otherwise.}
        \end{cases}
    \end{align*}
    
    A $\dfn{path}$ with $k$ steps consists of $k$ arcs $\gamma_{a_1, a_2}, \gamma_{a_2, a_3}, \dots, \gamma_{a_k, a_{k+1}}$ corresponding to the path of edges $v_{a_1} \to v_{a_2} \to \cdots \to v_{a_{k+1}}$. A $\dfn{cycle}$ is a path which starts and ends at the same boundary point. Note that, since arcs are clockwise directed around the puncture, all cycles must loop around the puncture, i.e., all cycles have nontrivial homotopy classes.
\end{defi}

\begin{defi}
    A \dfn{labeling} of a noncrossing arc diagram $A = (V, E)$ with $n$ boundary points is a map $\rho: V \to \Z$ such that $\rho(v_i) \equiv i \pmod{n}$ for all $1 \leq i \leq n$. Equivalently, we can view the labeling as picking a representative from the residue class of $i$ modulo $n$ for each boundary point $v_i$. The \dfn{identity labeling} is the map $\id: V \to \Z$ such that $\id(v_i) = i$ for all $1 \leq i \leq n$. 
\end{defi}

We first record how cycles in arc diagrams correspond to waning blocks.

\begin{lem}\label{lem:cycle}
    Given a $312$-avoiding TITO $\preceq$, its arc diagram $\cA(\preceq)$ includes a cycle if and only if $\preceq$ has a waning block. In this case, the residues modulo $n$ of the elements of the waning block are exactly the indices $i$ of the boundary points $v_i$ lying on the cycle. 
\end{lem}

\begin{proof}
    If $\cA(\preceq)$ includes a cycle $C = (\gamma_{a_1, a_2}, \dots, \gamma_{a_k, a_1})$ of $k$ steps with $a_1 < \cdots < a_k$, then $a_{1}+ bn \preceq a_1$ for some positive integer $b$, so $\preceq$ has a waning block. By Lemmas \ref{lem:312 blocks} and \ref{lem:windows of 312}, elements in this waning block must appear in decreasing order and the difference between elements in the window is less than $n$. Thus, every pair of consecutive elements forms a lower wall and is a part of the cycle $C$, and the waning block must be exactly $[\underline{a_k, \dots, a_1}]$. 

    Conversely, if a 312-avoiding TITO $\preceq$ has a waning block $[\underline{a_k, \dots, a_1}]$, then we know that $a_1+n \preceq \cdots \preceq a_k \preceq a_1$. Hence, the arcs $\gamma_{a_1, a_2}, \dots, \gamma_{a_k, a_{1}}$ form a cycle in $\cA(\preceq)$. 
\end{proof}

To construct the inverse map, we introduce the compass labeling of an arc diagram. We also define the walker set when the diagram contains a cycle and the cut-open diagram when it is acyclic.

\begin{defi}
    Let $A = (V, E)$ be a noncrossing arc diagram with $n$ boundary points. We define its compass labeling and, when $A$ contains a cycle, its walker set.
    
    If $A$ contains a cycle $C$, then let $C^{\circ}$ denote $C$ with its interior. The \dfn{walker set} $\mathsf{W}(A)$ is the set of connected components of the ambient annulus with $C^\circ$ removed that contain at least one marked boundary point.\footnote{For an example of a connected component without a boundary point, consider the connected component created by $\gamma_{10, 11}$ in Figure \ref{fig: compass cycle}.} Each component is equipped with the marked boundary points and arcs that it contains. We call the components in $\mathsf{W}(A)$ \dfn{walkers}. 
    
    Let $V(C)$ denote the boundary points in $C$ and $E(C)$ denote the edges in $C$. The \dfn{compass labeling} $\phi : V \to \Z$ is the unique labeling satisfying:    
    \begin{enumerate}
        \item $\phi(v_a) < \phi(v_b)$ for all arcs $\gamma_{a, b} \in E \setminus E(C)$;
        \item $\phi(v_i) > 0$ for all $v_i \in V $; 
        \item for all walkers $W \in \mathsf{W}(A)$, the set $\{\phi(v_i) : v_i \in V(W) \}$ forms an integer interval;
        \item for all walkers $W \in \mathsf{W}(A)$, the sum $\sum_{v_i \in V(W)} \phi(v_i)$ is minimized, and $\phi(v_i) = i$ for every $v_i \in V(C)$.
    \end{enumerate}

    Note that conditions $(2)$ and $(4)$ together guarantee that the compass labeling is unique; conditions $(1)$ and $(3)$ place no constraint on the boundary points lying on $C$, which is why $(4)$ fixes their labels directly. 

    For each walker, the minimum compass label among its boundary points uniquely determines it. We write $W_i$ for the walker whose minimum compass label is $i$.

    If $A$ does not contain a cycle, then the compass labeling of $A$ is determined by the same restrictions above, but with $E(C) = \emptyset$ in condition (1) and with $V$ in place of $V(W)$ in conditions (3) and (4).

\end{defi}

Each walker is acyclic, and every arc with an endpoint in a walker has both endpoints in that walker. Moreover, its compass labeling satisfies $\phi(v_a) < \phi(v_b)$ for every arc $\gamma_{a,b}$ in the walker.

For a walker with $m$ boundary points in an arc diagram with a cycle $C$, its compass labels are as follows. Let $v_i$ be the first boundary point of the walker encountered when traversing the diagram clockwise from the point in $C$ with the minimal compass label. The walker's boundary points are then assigned the labels $i,i+1,\dots,i+m-1$ in clockwise order, starting from $v_i$.

In all figures below, compass labels are shown in orange.

\begin{example}\label{ex: compass labeling}
        The compass labeling of the arc diagram for $[8, 7, 9,6,12,13,15,14,16][\underline{11, 10, 5}]$ is shown in Figure \ref{fig: compass cycle}.
        \begin{figure}[H]
            \centering
            \begin{tikzpicture}[
    dot/.style={circle, fill, inner sep=1.5pt},
    arrow/.style={->, blue, thick}
]
    \def\radius{3cm}

    \draw[thick, fill=gray!15] (0,0) circle (\radius);

    \draw[thick, fill=white, draw=black] (0,0) circle (5pt);

    \foreach \n in {1,2,3,4} {
        \pgfmathsetmacro{\angle}{90 - (\n * 30)}

        \pgfmathtruncatemacro{\num}{12 + \n}
        
        \node[dot, label={[align=center]\angle:$\compass{\num}$\\$v_{\n}$}] (\n) at (\angle:\radius) {};
    }

    \foreach \n in {5,...,12} {
        \pgfmathsetmacro{\angle}{90 - (\n * 30)}
        
        \node[dot, label={[align=center]\angle:$\compass{\n}$\\$v_{\n}$}] (\n) at (\angle:\radius) {};
    }

    \draw[arrow, looseness=1.2] (10) to[out=15, in=-100] (11);

    \draw[arrow, looseness=1.2] (11) to[out=-40, in=100] (5);

    \draw[arrow, looseness=0.2] (5) to[out=145, in=-65] (10);

    \draw[arrow, looseness=0.3] (2) to[out=-100, in=130] (3);
    
    \draw[arrow, looseness=1.0] (6) to[out=115, in=-25] (9);

    \draw[arrow, looseness=0.5] (7) to[out=120, in=-20] (8);
\end{tikzpicture}
\caption{The compass labeling of the noncrossing arc diagram of the TITO $[8, 7, 9,6,12,13,15,14,16][\underline{11, 10, 5}]$ with $n = 12$.}
\label{fig: compass cycle}
\end{figure}
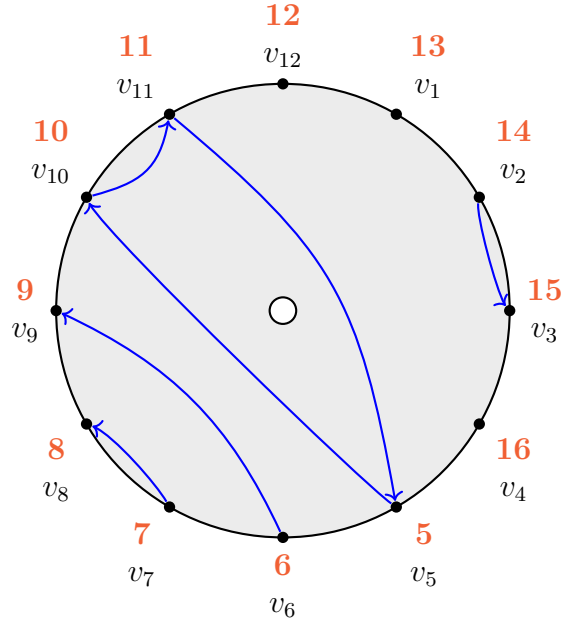

The walker set of this diagram, $\mathsf{W}(A) = \{W_{6}, W_{12}\}$, is shown in Figure \ref{fig: walker cycle}.
\begin{figure}[H]
    \centering
    \begin{tikzpicture}[
    dot/.style={circle, fill, inner sep=1.5pt},
    arrow/.style={blue, thick, -{Stealth[length=2mm, width=2mm]}}
]

\def\commonheight{2.2cm}

\begin{scope}
    \fill[gray!15] (-\commonheight, 0) -- (\commonheight, 0) arc (0:180:\commonheight);
    \draw[thick] (-\commonheight, 0) -- (\commonheight, 0);
    
    \draw[thick, dashed] (\commonheight, 0) arc (0:180:\commonheight);
    
    \node[dot, label=below:{$\compass{9}$}] (p9) at (-1.5,0) {};
    \node[dot, label=below:{$\compass{8}$}] (p8) at (-0.5,0) {};
    \node[dot, label=below:{$\compass{7}$}] (p7) at (0.5,0) {};
    \node[dot, label=below:{$\compass{6}$}] (p6) at (1.5,0) {};
    
    \draw[arrow] (p7.north) to[bend left=-60] (p8.north);
    \draw[arrow] (p6.north) to[bend left=-50] (p9.north);
\end{scope}
\begin{scope}[xshift=\commonheight * 2 + 2cm]
    \fill[gray!15] (-\commonheight, 0) -- (\commonheight, 0) arc (0:180:\commonheight);
    
    \draw[thick] (-\commonheight, 0) -- (\commonheight, 0);
    
    \draw[thick, dashed] (\commonheight, 0) arc (0:180:\commonheight);
    
    \node[dot, label=below:{$\compass{16}$}] (p16) at (-1.8,0) {};
    \node[dot, label=below:{$\compass{15}$}] (p15) at (-0.9,0) {};
    \node[dot, label=below:{$\compass{14}$}] (p14) at (0,0) {};
    \node[dot, label=below:{$\compass{13}$}] (p13) at (0.9,0) {};
    \node[dot, label=below:{$\compass{12}$}] (p12) at (1.8,0) {};
    
    \draw[arrow] (p14.north) to[bend left=-60] (p15.north);
\end{scope}

\end{tikzpicture}
    \caption{The walker set of the arc diagram in Figure \ref{fig: compass cycle}.}
    \label{fig: walker cycle}
\end{figure}
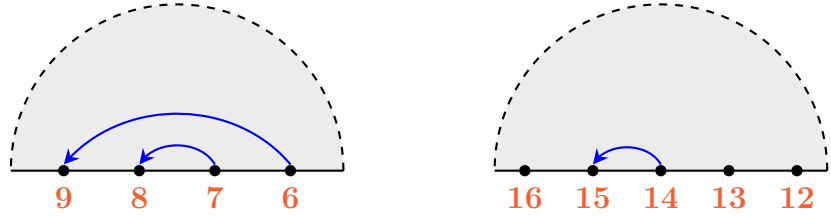
\end{example}

For an acyclic noncrossing arc diagram, we use its compass labeling to cut open the annulus.

\begin{defi}
For an acyclic arc diagram $A = (V, E)$ with $n$ boundary points and compass labeling $\phi$, the \dfn{cut-open diagram} of $A$ is obtained by making a cut between the boundary points with maximum and minimum compass labels.

More precisely, choose a point on the boundary segment from the point with maximum compass label to the point with minimum compass label. Draw a simple curve from this point to the inner boundary, with its interior contained in the interior of the annulus and disjoint from every arc. Cut the annulus along this curve, leaving all arcs and marked points intact. The resulting region is homeomorphic to a rectangle. The former outer boundary becomes a boundary segment on which the marked points occur in increasing compass-label order. 

We retain the compass labeling of the original arc diagram. 
\end{defi}

\begin{example}\label{ex: acyclic cut}
    Take $n=4$ and the acyclic arc diagram with the single arc $\gamma_{2,3}$. The compass labeling is the identity. Choose the slit in the clockwise gap from $v_4$ to $v_1$, as shown in Figure~\ref{fig: acyclic cut example}. After cutting open the annulus, the marked boundary points occur in the order $v_1,v_2,v_3,v_4$.

    \begin{figure}[H]
    \centering
    \scalebox{0.95}{
    \begin{tikzpicture}[
        dot/.style={circle, fill, inner sep=1.5pt},
        arrow/.style={
            blue, thick,
            -{Stealth[length=2mm, width=2mm]}
        }
    ]

    \begin{scope}
        \def\R{2.45cm}
        \def\r{0.25cm}

        \draw[thick, fill=gray!15] (0,0) circle (\R);
        \draw[thick, fill=white] (0,0) circle (\r);

        \node[dot, label={[align=center]90:
            $\compass{1}$\\$v_1$}]
            (v1) at (90:\R) {};

        \node[dot, label={[align=center]0:
            $\compass{2}$\\$v_2$}]
            (v2) at (0:\R) {};

        \node[dot, label={[align=center]270:
            $\compass{3}$\\$v_3$}]
            (v3) at (270:\R) {};

        \node[dot, label={[align=center]180:
            $\compass{4}$\\$v_4$}]
            (v4) at (180:\R) {};

        \draw[arrow]
            (v2) to[out=-110, in=20, looseness=0.95] (v3);

        \draw[thick, dashed] (132:\r) -- (132:\R);
    \end{scope}

    \node at (3.85,0) {$\longrightarrow$};

    \begin{scope}[xshift=5.35cm, yshift=-1.325cm]
        \def\W{6.6cm}
        \def\H{2.65cm}

        \fill[gray!15] (0,0) rectangle (\W,\H);

        \draw[thick] (0,\H) -- (\W,\H);
        \draw[thick] (0,0) -- (\W,0);

        \draw[thick, dashed] (0,0) -- (0,\H);
        \draw[thick, dashed] (\W,0) -- (\W,\H);

        \node[dot, label={[align=center]90:
            $\compass{1}$\\$v_1$}]
            (w1) at (0.85,\H) {};

        \node[dot, label={[align=center]90:
            $\compass{2}$\\$v_2$}]
            (w2) at (2.40,\H) {};

        \node[dot, label={[align=center]90:
            $\compass{3}$\\$v_3$}]
            (w3) at (4.00,\H) {};

        \node[dot, label={[align=center]90:
            $\compass{4}$\\$v_4$}]
            (w4) at (5.55,\H) {};

        \draw[arrow]
            (w2.south) to[bend right=42] (w3.south);
    \end{scope}

    \end{tikzpicture}
    }
    \caption{The acyclic diagram with the single arc $\gamma_{2,3}$
    and the diagram obtained by cutting open the annulus along the dashed segment.}
    \label{fig: acyclic cut example}
\end{figure}
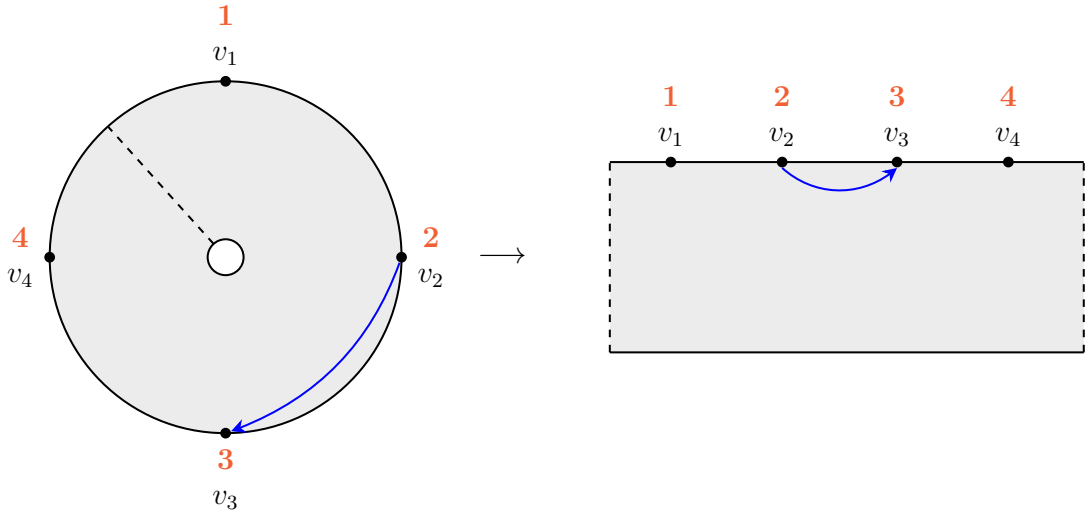
\end{example}

\begin{example}
    The compass labeling of the arc diagram for $[8,7,9,6,11,12,15,14,16,13,10,5]$ is shown in Figure \ref{fig: compass no cycle}.

\begin{figure}[H]
            \centering
            \begin{tikzpicture}[
    dot/.style={circle, fill, inner sep=1.5pt},
    arrow/.style={->, blue, thick}
]
    \def\radius{3cm}

    \draw[thick, fill=gray!15] (0,0) circle (\radius);

    \draw[thick, fill=white, draw=black] (0,0) circle (5pt);

    \foreach \n in {1,2,3,4} {
        \pgfmathsetmacro{\angle}{90 - (\n * 30)}

        \pgfmathtruncatemacro{\num}{12 + \n}
        
        \node[dot, label={[align=center]\angle:$\compass{\num}$\\$v_{\n}$}] (\n) at (\angle:\radius) {};
    }

    \foreach \n in {5,...,12} {
        \pgfmathsetmacro{\angle}{90 - (\n * 30)}
        
        \node[dot, label={[align=center]\angle:$\compass{\n}$\\$v_{\n}$}] (\n) at (\angle:\radius) {};
    }

    \draw[arrow, looseness=1.2] (10) to[out=20, in=190] (1);

    \draw[arrow, looseness=1.2] (1) to[out=-70, in=100] (4);

    \draw[arrow, looseness=0.2] (5) to[out=145, in=-65] (10);

    \draw[arrow, looseness=0.3] (2) to[out=-100, in=130] (3);
    
    \draw[arrow, looseness=1.0] (6) to[out=115, in=-25] (9);

    \draw[arrow, looseness=0.5] (7) to[out=120, in=-20] (8);

\end{tikzpicture}
            \caption{The compass labeling of the arc diagram of the TITO $[8, 7, 9, 6, 11, 12, 15, 14, 16, 13, 10, 5]$ with $n = 12$.}
            \label{fig: compass no cycle}
        \end{figure}
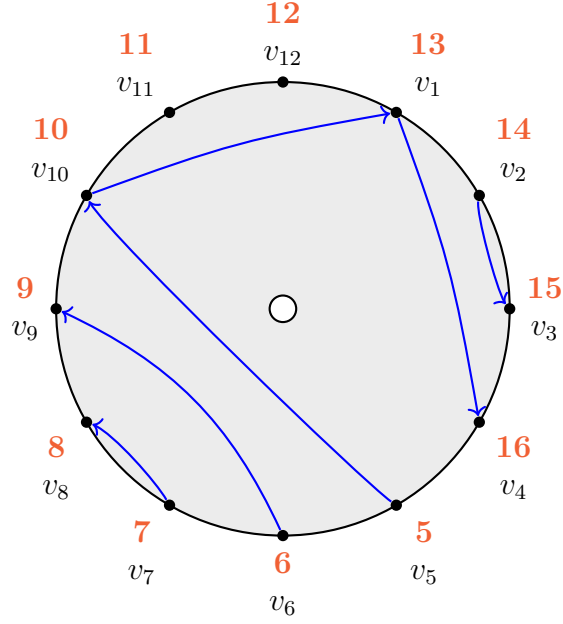

    Here the compass labels are $5,6,\dots,16$. The minimum label is $\phi(v_5)=5$, and the maximum label is $\phi(v_4)=16$. Thus, we make the cut in the clockwise gap from $v_4$ to $v_5$.
\end{example}

We next introduce the subcompass path and subwalker set used in the recursive construction. We first return to the diagrams in Figures \ref{fig:arcdiagram1} and \ref{fig:arcdiagram2} to describe their walkers.

\begin{example}\label{ex:connected components}
    In Figure \ref{fig:arcdiagram1}, the TITO is $[1, 4, 3, 6, 5][\underline{7, 2}]$, its cycle is $C = (2, 7)$, and $\mathsf{W}(A) = \{W_1, W_3\}$; the walkers, with their compass labeling, are shown in Figure \ref{fig:connected components1}. 
    \begin{figure}[H]
        \centering
        \scalebox{1}{\begin{tikzpicture}[
    dot/.style={circle, fill, inner sep=1.5pt},
]

\def\commonheight{2cm}

\begin{scope}
    \fill[gray!15] (-\commonheight, 0) -- (\commonheight, 0) arc (0:180:\commonheight);
    \draw[thick] (-\commonheight, 0) -- (\commonheight, 0);
    
    \draw[thick, dashed] (\commonheight, 0) arc (0:180:\commonheight);
    
    \node[dot, label=below:$\compass{1}$] at (0,0) {};
\end{scope}

\begin{scope}[xshift=\commonheight * 2 + 2cm]
    \fill[gray!15] (-\commonheight, 0) -- (\commonheight, 0) arc (0:180:\commonheight);
    \draw[thick] (-\commonheight, 0) -- (\commonheight, 0);
    
    \draw[thick, dashed] (\commonheight, 0) arc (0:180:\commonheight);
    
    \node[dot, label=below:$\compass{6}$] (p6) at (-1.5,0) {};
    \node[dot, label=below:$\compass{5}$] (p5) at (-0.5,0) {};
    \node[dot, label=below:$\compass{4}$] (p4) at (0.5,0) {};
    \node[dot, label=below:$\compass{3}$] (p3) at (1.5,0) {};
    
    \draw[blue, thick, -{Stealth[length=2mm, width=2mm]}] (p5.north) to[bend right=70] (p6.north);
    \draw[blue, thick, -{Stealth[length=2mm, width=2mm]}] (p3.north) to[bend right=70] (p4.north);
\end{scope}

\end{tikzpicture}}

        \caption{The elements of $\mathsf{W}(A)$ for $A  = \cA(\preceq)$, where $\preceq$ is $[1, 4, 3, 6, 5][\underline{7, 2}]$.}
        \label{fig:connected components1}
    \end{figure}
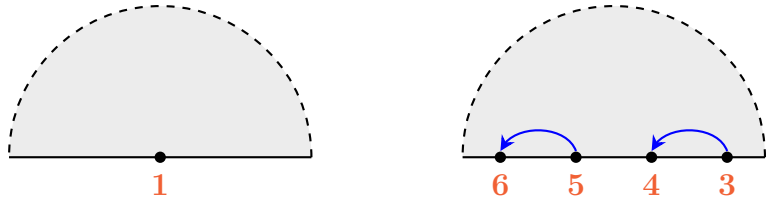

    In Figure \ref{fig:arcdiagram2}, the TITO is $[1, 6, 5, 7, 4, 3][\underline{2}]$, the cycle is $C = (2)$, and $\mathsf{W}(A) = \{W_3\}$, where the compass-labeled walker is shown in Figure \ref{fig:connected components2}. Note that here $\phi(v_1) = 8$. (Otherwise, $\phi(v_3), \phi(v_4),\dots, \phi(v_1)$ do not form an integer interval.)
    \begin{figure}[H]
        \centering
        \scalebox{0.7}{\begin{tikzpicture}[
    dot/.style={circle, fill, inner sep=1.5pt},
]

    \fill[gray!15] (-1, 0) -- (6, 0) arc (0:180:3.5);
\node[dot, label=below:$\compass{8}$] (p8) at (0,0) {};
\node[dot, label=below:$\compass{7}$] (p7) at (1,0) {};
\node[dot, label=below:$\compass{6}$] (p6) at (2,0) {};
\node[dot, label=below:$\compass{5}$] (p5) at (3,0) {};
\node[dot, label=below:$\compass{4}$] (p4) at (4,0) {};
\node[dot, label=below:$\compass{3}$] (p3) at (5,0) {};

\draw[thick] (-1,0) -- (6,0);

\draw[thick, dashed] (6,0) arc (0:180:3.5);

\draw[blue, thick, -{Stealth[length=2mm, width=2mm]}] (p5.north) to[bend right=50] (p6.north);
\draw[blue, thick, -{Stealth[length=2mm, width=2mm]}] (p4.north) to[bend right=40, looseness=1.2] (p7.north);
\draw[blue, thick, -{Stealth[length=2mm, width=2mm]}] (p3.north) to[bend right=50] (p4.north);
\end{tikzpicture}}
        \caption{The elements of $\mathsf{W}(A )$ for $A = \cA(\preceq)$, where $\preceq$ is $[1, 6, 5, 7, 4, 3][\underline{2}]$.}
        \label{fig:connected components2}
    \end{figure}
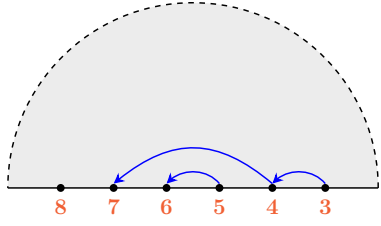
\end{example}

We next introduce the subcompass path used in the recursive construction.

\begin{defi}
    Given a walker or a cut-open diagram with at least one arc, choose the boundary point with smallest compass label among those with an outgoing arc, and let $P=(v_{j_1},\ldots,v_{j_{k+1}})$ be the longest path starting at this point. We call $P$ the \dfn{subcompass path}.

    Choose a point $p$ on the open boundary, and complete $P$ to a cycle $P'$ by adjoining arcs from $v_{j_{k+1}}$ to $p$ and from $p$ to $v_{j_1}$. Let $P^\circ$ denote $P'$ together with its interior. The connected components of the complement of $P^\circ$ that contain at least one boundary point, equipped with the boundary points, arcs, and inherited compass labeling that they contain, are called \dfn{subwalkers}.
\end{defi}

\begin{example}
    For the walker $W_3$ in Figure~\ref{fig:connected components2}, the subcompass path separates the remaining boundary points into the two subwalkers shown in Figure~\ref{fig: subwalkers}, whose minimum compass labels are $5$ and $8$. These are the connected components of $W_3 \setminus P^{\circ}$ that contain boundary points, where $P^{\circ}$ is the closed region bounded by the green and maroon arcs in Figure~\ref{fig: P circ}.
    
    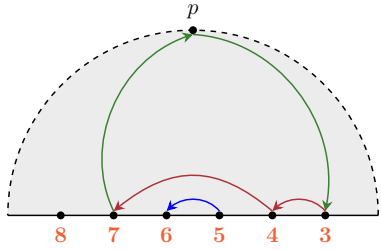
\begin{figure}[H]
        \centering
        \scalebox{0.7}{\begin{tikzpicture}[
    dot/.style={circle, fill, inner sep=1.5pt},
]
    \fill[gray!15] (-1, 0) -- (6, 0) arc (0:180:3.5);

\node[dot, label=above:$p$] (p0) at (2.5,3.5) {};
\node[dot, label=below:$\compass{8}$] (p8) at (0,0) {};
\node[dot, label=below:$\compass{7}$] (p7) at (1,0) {};
\node[dot, label=below:$\compass{6}$] (p6) at (2,0) {};
\node[dot, label=below:$\compass{5}$] (p5) at (3,0) {};
\node[dot, label=below:$\compass{4}$] (p4) at (4,0) {};
\node[dot, label=below:$\compass{3}$] (p3) at (5,0) {};

\draw[thick] (-1,0) -- (6,0);

\draw[thick, dashed] (6,0) arc (0:180:3.5);

\draw[OliveGreen, thick, -{Stealth[length=2mm, width=2mm]}] (p7.north) to[bend left=50] (p0.south);

\draw[OliveGreen, thick, -{Stealth[length=2mm, width=2mm]}] (p0.south) to[bend left=50] (p3.north);

\draw[blue, thick, -{Stealth[length=2mm, width=2mm]}] (p5.north) to[bend right=50] (p6.north);
\draw[Maroon, thick, -{Stealth[length=2mm, width=2mm]}] (p4.north) to[bend right=40, looseness=1.2] (p7.north);
\draw[Maroon, thick, -{Stealth[length=2mm, width=2mm]}] (p3.north) to[bend right=50] (p4.north);
\end{tikzpicture}}
        \caption{The cycle $P'$ of the walker in Figure \ref{fig:connected components2} is colored in green and maroon, where the maroon edges denote the subcompass path $P$.}
        \label{fig: P circ}
    \end{figure}
    
    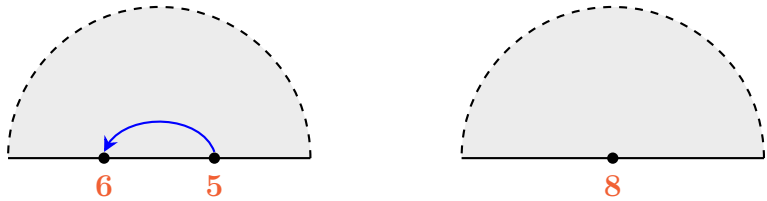
\begin{figure}[H]
        \centering
        \scalebox{1}{\begin{tikzpicture}[
    dot/.style={circle, fill, inner sep=1.5pt},
]

\def\commonheight{2cm}
\begin{scope}
    \fill[gray!15] (-\commonheight, 0) -- (\commonheight, 0) arc (0:180:\commonheight);
    \draw[thick] (-\commonheight, 0) -- (\commonheight, 0);
    
    \draw[thick, dashed] (\commonheight, 0) arc (0:180:\commonheight);
    
    \node[dot, label=below:$\compass{6}$] (p6) at (-0.73,0) {};
    \node[dot, label=below:$\compass{5}$] (p5) at (0.73,0) {};
    
    \draw[blue, thick, -{Stealth[length=2mm, width=2mm]}] (p5.north) to[bend right=70] (p6.north);
\end{scope}
\begin{scope}[xshift=\commonheight * 2 + 2cm]
    \fill[gray!15] (-\commonheight, 0) -- (\commonheight, 0) arc (0:180:\commonheight);
    \draw[thick] (-\commonheight, 0) -- (\commonheight, 0);
    
    \draw[thick, dashed] (\commonheight, 0) arc (0:180:\commonheight);
    
    \node[dot, label=below:$\compass{8}$] at (0,0) {};
\end{scope}

\end{tikzpicture}}
        \caption{Subwalkers of the walker in Figure \ref{fig:connected components2}.}
        \label{fig: subwalkers}
    \end{figure}
\end{example}

The recursive algorithm below applies to a walker or a cut-open diagram equipped with its compass labeling. For a cut-open diagram, we form the subwalkers using the compass-label intervals specified in Algorithm \ref{algo:support}.

\subsection{Constructing TITOs from noncrossing arc diagrams}\label{subsec:inverse construction}As mentioned, Barkley and Defant \cite{BD25+} showed how to
construct the arc diagram of a TITO and, specializing a theorem of Barkley \cite{B25+}, proved that this gives a bijection between $312$-avoiding TITOs and noncrossing arc diagrams. We now give the inverse construction. Given a noncrossing arc diagram, we build the corresponding $312$-avoiding TITO as follows.
\begin{enumerate}
    \item If the diagram is acyclic, apply the recursive subarray construction to its cut-open diagram. This produces the window of a single waxing block.

    \item If the diagram contains a cycle, write the compass labels of the boundary points on the cycle in decreasing order. This produces the window of the waning block. Form the walker set and apply the recursive subarray construction to each walker.
    Concatenate the resulting subarrays in increasing order of minimum compass label. This produces the window of the waxing block, if one exists.
\end{enumerate}

Let $I$ be a waxing block represented in window notation as $[a_1, \dots, a_{|I|}]$. We say that a word is a \dfn{waxing window subarray} if it is a contiguous subarray of $[a_1, \dots, a_{|I|}]$.

We first describe the recursive subarray construction, which produces a $312$-avoiding waxing window subarray from a walker or a cut-open diagram and its compass labeling. The construction begins with the subcompass path, writes its compass labels in decreasing order, and then recursively constructs and concatenates the subarrays arising from the subwalkers.

\begin{algo}\label{algo:support}
    Let $W_j$ be either a walker or a cut-open diagram. Let $\phi$ be its compass labeling. Write its $m$ boundary points as $v_{i+1},\dots,v_{i+m}$ in clockwise order, with vertex subscripts taken modulo $n$.\footnote{One can check their compass labels are consecutive and increasing, and every arc runs from a smaller compass label to a larger one.} We construct a waxing window subarray $\mathcal{W}_j$ as follows.
    \begin{enumerate}[label=(\arabic*)]
        \item If $E(W_j)=\emptyset$, return $\mathcal{W}_j=[\phi(v_{i+1}),\dots,\phi(v_{i+m})].$

        \item Otherwise, let $P=(v_{j_1},\dots,v_{j_{k+1}})$ be the subcompass path.

        \item Initialize $\mathcal{W}_j= [\phi(v_{j_{k+1}}),\phi(v_{j_k}),\dots,\phi(v_{j_1})].$

        \item Form the subwalkers from the nonempty compass-label intervals
        $$[\phi(v_{i+1}),\phi(v_{j_1})-1], \quad[\phi(v_{j_t})+1,\phi(v_{j_{t+1}})-1]
                \quad (1\le t\le k), \quad [\phi(v_{j_{k+1}})+1,\phi(v_{i+m})].
        $$

        \item For each such subwalker, let $W_\ell'$ denote the one whose minimum compass label is $\ell$. Set
        $$
            \mathcal{L}=\{W_\ell':\ell<\phi(v_{j_{k+1}})\}, \qquad \mathcal{R}=\{W_\ell':\ell>\phi(v_{j_{k+1}})\}.
        $$

        \item Process the subwalkers in $\mathcal{L}$ in decreasing order of minimum compass label. For each one, apply this algorithm recursively, and concatenate the resulting subarray to the left of $\mathcal{W}_j$.

        \item Process the subwalkers in $\mathcal{R}$ in increasing order of minimum compass label. For each one, apply this algorithm recursively, and concatenate the resulting subarray to the right of $\mathcal{W}_j$.

        \item Return $\mathcal{W}_j$.
    \end{enumerate}
\end{algo}

We now prove the correctness of our algorithm by showing that the lower walls of $\mathcal{W}_j$ correspond exactly to the arcs in $W_j$ and that $\mathcal{W}_j$ is 312-avoiding.

\begin{proof}
We prove both claims by strong induction on $|E(W_j)|$, where $W_j$ is either a walker or a cut-open diagram. 

When $E(W_j) = \emptyset$, we return $\mathcal{W}_j = [\phi(v_{i+1}),\dots,\phi(v_{i+m})]$. Its entries are strictly increasing and contain each compass label exactly once.
Thus, $\mathcal{W}_j$ has no lower walls and is $312$-avoiding.

Now suppose $E(W_j) \neq \emptyset$, and assume the assertions hold for every input with fewer arcs than $W_j$. Let $P = (v_{j_1},\dots,v_{j_{k+1}})$ be the subcompass path, and set $a_t = \phi(v_{j_t})$ for $1 \leq t \leq k+1$. Since compass labels increase along arcs, we have $a_1 < \cdots < a_{k+1}$.

The initial vertex of $P$ has no incoming arc by the choice of $a_1$, and the final vertex has no outgoing arc by the maximality of $P$.
Since each boundary point has at most one incoming arc and at most one outgoing arc, no arc outside $P$ has an endpoint on $P$.

The subwalkers are formed from the nonempty integer intervals
    $$[\phi(v_{i+1}),a_1-1], \qquad [a_t+1,a_{t+1}-1] \quad (1 \leq t \leq k), \qquad [a_{k+1}+1,\phi(v_{i+m})],$$
with the inherited arcs and compass labeling. Every remaining arc has both endpoints in one interval. Indeed, an arc joining different intervals would either have its source at a label smaller than $a_1$ or cross an arc of $P$, contradicting the choice of $a_1$ or noncrossing.

Thus, every arc of $W_j$ lies in $P$ or in a single subwalker. Each subwalker inherits the hypotheses of the algorithm and has fewer arcs than $W_j$. The intervals above, together with the labels on $P$, partition the label set of $W_j$.

The nonempty ones among the first $k+1$ intervals give the subwalkers in $\mathcal{L}$, and the last interval, if nonempty, gives the subwalker in $\mathcal{R}$. Every label in one interval is smaller than every label in a later interval.

From left to right, $\mathcal{W}_j$ consists of the subarrays from $\mathcal{L}$ in increasing order of minimum label, followed by $[a_{k+1},\dots,a_1]$, followed by the subarrays from $\mathcal{R}$ in increasing order of minimum label. For $\mathcal{L}$, this order follows because the subarrays are prepended in decreasing order of minimum label.

We now show that the lower walls of $\mathcal{W}_j$ correspond exactly to the arcs of $W_j$. By the inductive hypothesis, the lower walls within each recursive subarray correspond exactly to the arcs of its subwalker. The lower walls within $[a_{k+1},\dots,a_1]$ correspond exactly to the arcs of $P$. It remains to check the points of concatenation.

Between consecutive recursive subarrays, every label in the earlier one is smaller than every label in the later one. The subarray arising from $P$ begins with $a_{k+1}$, which is greater than every label in the subarrays from $\mathcal{L}$, and ends with $a_1$, which is smaller than every label in the subarrays from $\mathcal{R}$.
Thus, no additional lower wall occurs at a point of concatenation.

It remains to show that $\mathcal{W}_j$ is $312$-avoiding. Suppose for contradiction that entries $x,y,z$, in this order in $\mathcal{W}_j$, satisfy $x > z > y$. We consider the subarray containing $y$.

Suppose $y$ lies in a subarray $S$ arising from $\mathcal{L}$. Since all earlier subarrays have smaller labels, $x$ must also lie in $S$.
The entry $z$ cannot lie in $S$ by the inductive hypothesis or in a later recursive subarray, whose labels are all larger than $x$.
Thus, $z$ lies in $[a_{k+1},\dots,a_1]$. However, no label of $P$ lies between two labels of $S$, contradicting $y < z < x$.

Suppose $y$ lies in $[a_{k+1},\dots,a_1]$. If $z$ lies later in this strictly decreasing subarray, then $z < y$. Otherwise, $z$ lies in a subarray from $\mathcal{R}$, so $z > a_{k+1} \geq x$. Both possibilities contradict $x > z > y$.

Finally, suppose $y$ lies in a subarray $S$ arising from $\mathcal{R}$. All earlier subarrays have smaller labels, so $x$ lies in $S$. All later subarrays have larger labels, so $z$ also lies in $S$. This contradicts the inductive hypothesis.

Therefore, $\mathcal{W}_j$ is $312$-avoiding, as desired.
\end{proof}

\begin{example}\label{ex: walker}
    Applying Algorithm~\ref{algo:support} to the walker shown in Figure \ref{fig:connected components2}, we obtain the subcompass path $P = (v_3, v_4, v_7)$. Thus, the initial subarray is $[7,4,3]$.

    The window subarrays we obtain from subwalkers in $\{W_5', W_8'\}$ are $[6, 5]$ and $[8]$, respectively (see Figure \ref{fig: subwalkers}). Concatenation yields the window subarray $[6,5,7,4,3,8]$.
\end{example}

\begin{example}
    Applying Algorithm~\ref{algo:support} to the cut-open diagram shown in Figure \ref{fig: acyclic cut example}, we obtain the subcompass path $P=(v_2,v_3)$. Thus, the initial subarray is $[3,2]$.

    The window subarrays we obtain from subwalkers in $\{W_1', W_4'\}$ are $[1]$ and $[4]$, respectively.
    Concatenation yields the window $[1,3,2,4].$
    
    The cut separates the labels $4$ and $1$, so they are processed on opposite sides of the subcompass-path subarray. This example shows why, in the acyclic case, one must first cut the annulus at the seam between the maximum and minimum compass labels and then apply the recursive construction to the resulting diagram.
\end{example}

Recall that 312-avoiding TITOs correspond exactly to noncrossing arc diagrams by Theorem \ref{thm:arcdiagram312}. We are now ready to describe the general algorithm to construct a 312-avoiding TITO from a noncrossing arc diagram $A$.

\begin{algo}\label{algo: inverse construction}
    Given a noncrossing arc diagram $A$ with compass labeling $\phi$, our algorithm returns the window notation of the corresponding $312$-avoiding TITO.
    \begin{enumerate}[label=(\arabic*)]
        \item Initialize $\mathcal{W}_{\mathrm{wan}}=\mathcal{W}_{\mathrm{wax}}=[].$

        \item If $A$ is acyclic, let $m=\min\{\phi(v_i):v_i\in V(A)\}.$
        \begin{enumerate}[label=(\alph*)]
            \item Cut the diagram between the boundary points with compass labels $m+n-1$ and $m$ to form the cut-open diagram.
            \item Apply Algorithm~\ref{algo:support} to the cut-open diagram. Let $\mathcal{W}_{\mathrm{wax}}$ be the resulting subarray. Return the TITO with this single waxing window.
        \end{enumerate}

        \item Otherwise, let $C$ be the cycle of $A$.

        \item Form $\mathcal{W}_{\mathrm{wan}}$ by writing the compass labels of the boundary points of $C$ in decreasing order.

        \item For each walker $W_i\in\mathsf{W}(A)$, apply Algorithm~\ref{algo:support}, and denote its output by $\mathcal{W}_i$.

        \item In increasing order of subscript, concatenate the subarrays $\mathcal{W}_i$ to the right of $\mathcal{W}_{\mathrm{wax}}$.

        \item Return the TITO given by $\mathcal{W}_{\mathrm{wax}}\mathcal{W}_{\mathrm{wan}}$.
    \end{enumerate}
\end{algo}

We prove the correctness of our algorithm by showing that the TITO $\preceq$ we construct has $A=\mathcal{A}(\preceq)$ and is $312$-avoiding.

\begin{proof}
    We consider two cases depending on whether $A$ contains a cycle.

    First, suppose $A$ does not contain a cycle. By the correctness of Algorithm \ref{algo:support}, the returned subarray $\mathcal{W}_{\text{wax}}$ contains each compass label exactly once, is $312$-avoiding, and has lower walls corresponding exactly to the arcs of $A$. Its entries represent every residue class modulo $n$, so it defines a single waxing window. Moreover, since these entries are $n$ consecutive integers, every entry in one translate of the window is smaller than every entry in the next translate. Thus, no lower wall occurs between successive translates. Any $312$ pattern must also lie in a single translate, since its first entry is its largest. Therefore, $A = \mathcal{A}(\preceq)$ and $\preceq$ is $312$-avoiding. This includes the case $E(A) = \emptyset$, for which the returned window is $[1,2,\dots,n]$.

    Now, suppose $A$ contains a cycle $C$. For each $v_i \in V(C)$, the compass labeling $\phi(v_i) = i$ by definition. Writing these labels in decreasing order gives a strictly decreasing waning block. Its lower walls correspond exactly to the arcs of $C$, including the lower wall between successive translates of its window. If the waxing window is empty, then the resulting TITO is strictly decreasing, so both claims hold.

    Suppose the waxing window is nonempty. Each walker occupies a clockwise gap between successive vertices of $C$. Its compass labels form an integer interval and increase in clockwise order.
    Moreover, all walker labels lie in an interval of $n$ consecutive integers. If $v_1$ or $v_n$ lies on $C$, then all compass labels lie in $[1,n]$. Otherwise, let $c$ be the smallest index of a vertex on $C$. The walker containing $v_n$ and $v_1$ receives labels extending through $n+c-1$, and all compass labels lie in $[c,n+c-1]$.

    We first show $A = \mathcal{A}(\preceq)$. By the correctness of Algorithm \ref{algo:support}, the lower walls within each $\mathcal{W}_i$ correspond exactly to the arcs of $W_i$. If $a < b$, then every label in $\mathcal{W}_a$ is smaller than every label in $\mathcal{W}_b$. Thus, concatenating these subarrays in increasing order of subscript does not create a lower wall at any point of concatenation. Since all entries of $\mathcal{W}_{\text{wax}}$ lie in an interval of $n$ consecutive integers, no lower wall occurs between successive translates of the waxing window either. Finally, there are no cover relations between the two blocks, since the waxing block has no maximal element and the waning block has no minimal element. Together with the lower walls of the waning block, this gives $A = \mathcal{A}(\preceq)$.

    We now prove that $\preceq$ is $312$-avoiding. Within the waxing block, the subarrays $\mathcal{W}_i$ and their translates occur in increasing order of their disjoint label intervals. Since the first entry of a $312$ pattern is its largest, any such pattern must lie in a single translated subarray. This is impossible by the correctness of Algorithm \ref{algo:support}. The waning block is strictly decreasing, so it also avoids $312$.

    Finally, suppose for contradiction that a $312$ pattern uses both blocks. Write its entries in TITO order as $x \preceq y \preceq z$, where $x > z > y$. Since the waning block is decreasing and follows the waxing block, $x$ and $y$ must lie in the waxing block, and $z$ must lie in the waning block. The inversion $x > y$ forces $x$ and $y$ to lie in the same translated subarray $\mathcal{W}_i$. The labels in this subarray form an integer interval, so every integer between $y$ and $x$ belongs to the waxing block. This contradicts $y < z < x$ with $z$ in the waning block.

    Thus, the returned TITO is $312$-avoiding and has arc diagram $A$.
\end{proof}

\begin{example}
    We apply the algorithm to the noncrossing arc diagram shown in Figure \ref{fig:arcdiagram2}. From the cycle, we get the waning block $[\underline{2}]$. From the walkers, we obtain the waxing window $[6,5,7,4,3,8]$ shown in Example \ref{ex: walker}. Thus, the TITO $[6,5,7,4,3,8][\underline{2}]$ corresponds to the noncrossing arc diagram in Figure \ref{fig:arcdiagram2}, as desired. Note that this is the same TITO as $[1, 6, 5, 7, 4, 3][\underline{2}]$, written using a different choice of window for its waxing block.
\end{example}

We apply Algorithm \ref{algo: inverse construction} in \S\ref{subsubsec: arc diagram approach} to enumerate $(312,p)$-avoiding TITOs for $p \in \{3214, 2314\}$, by translating pattern-avoidance conditions on the TITO into structural conditions on its noncrossing arc diagram.

\section{TITOs avoiding $p, q \in S_3$}\label{sec: results length 3}

In this section, we count TITOs avoiding any two length $3$ patterns. Although there are $\binom{6}{2} = 15$ pairs of patterns, we need to consider only $6$ cases. This reduction follows from the antiautomorphisms $\Psi_{\leftrightarrow}$ and $\Psi_{\updownarrow}$, where $\Psi_{\leftrightarrow}$ reverses each TITO and $\Psi_{\updownarrow}$ is induced by the map $\Z \to \Z$ taking $x \mapsto -x$. Applying these antiautomorphisms shows that $f_n^{(p,q)}$ depends only on the following $6$ cases, up to relabeling $(p,q)$:
\begin{enumerate}[label=(\alph*)]
    \item $f_n^{(312, 123)} = f_n^{(213, 321)} = f_n^{(132, 321)} = f_n^{(231, 123)}$;
    \item $f_n^{(312, 132)} = f_n^{(213, 231)}$;
    \item $f_n^{(312, 231)} = f_n^{(213, 132)}$;
    \item $f_n^{(312, 213)} = f_n^{(132, 231)}$;
    \item $f_n^{(312, 321)} = f_n^{(213, 123)} = f_n^{(132, 123)} = f_n^{(231, 321)}$;
    \item $f_n^{(123, 321)}$.
\end{enumerate}
We treat each case separately, and prove the following theorem.

\begin{thm}\label{thm: enumerate 2 length 3}
    For $n \geq 1$, the cases above have the following enumerations:
    \begin{enumerate}[label=(\alph*)]
        \item $f_n^{(312, 123)} = f_n^{(213, 321)} = f_n^{(132, 321)} = f_n^{(231, 123)} = 1$;
        \item $f_n^{(312, 132)} = f_n^{(213, 231)} = 2$;
        \item $f_n^{(312, 231)} = f_n^{(213, 132)} = 2^n$;
        \item $f_n^{(312, 213)} = f_n^{(132, 231)} = 2^n$;
        \item $f_n^{(312, 321)} = f_n^{(213, 123)} = f_n^{(132, 123)} = f_n^{(231, 321)} = 2^n - 1$;
        \item $f_n^{(123, 321)} = 0$.
    \end{enumerate}
\end{thm}

We consider the cases in the order they are listed above.

\begin{prop}
    The number of $(312, 123)$-avoiding TITOs is $f_n^{(312, 123)}= 1.$ Specifically, the only $(312, 123)$-avoiding $n$-TITO is $[\underline{n, \dots, 1}]$.
\end{prop}
\begin{proof}
    A $(312, 123)$-avoiding TITO cannot have any waxing blocks. By Lemma \ref{lem:312 blocks}, the TITO must be $[\underline{n, \dots, 1}]$, which is $(312, 123)$-avoiding. Thus, $f_n^{(312, 123)} = 1.$
\end{proof}

\begin{prop}
    The number of $(312, 132)$-avoiding TITOs is $f_n^{(312, 132)}= 2.$ Specifically, the only $(312, 132)$-avoiding $n$-TITOs are $[\underline{n, \dots, 1}]$ and $[1, \dots, n]$.
\end{prop}
\begin{proof}    
    A TITO avoiding $132$ cannot have two blocks such that the first is waxing and the second is waning. Thus, by Lemma \ref{lem:312 blocks}, any $(312 , 132)$-avoiding TITO consists of exactly one block. If the block is waning, then the TITO must be $[\underline{n, \dots, 1}]$. If the block is waxing, then the terms in the block must be increasing to avoid $132$, so the TITO must be $[1, \dots, n]$. Since $[\underline{n, \dots, 1}]$ and $[1, \dots, n]$ avoid $312$ and $132$, we have $f_n^{(312, 132)} = 2$.
\end{proof}

\begin{prop}\label{prop: 312 231}
    The number of $(312, 231)$-avoiding TITOs is $f_n^{(312, 231)} = 2^n.$
\end{prop}
\begin{proof}
    A TITO avoiding $231$ cannot have two blocks such that the first is waxing and the second is waning. Thus, by Lemma \ref{lem:312 blocks}, any $(312, 231)$-avoiding TITO consists of exactly one block, which is either waxing or waning. If the block is waning, then the TITO must be $[\underline{n, \dots, 1}]$, which is $(312, 231)$-avoiding. If the block is waxing, then the TITO is an affine permutation, so it suffices to compute $a_n^{(312, 231)}$, the number of such affine permutations.
    
    By Lemma \ref{lem:inversion}, we have $a_n^{(312, 231)} = a_n^{(231, 312)}$. Applying Lemma \ref{lem: formula affine avoid n---} to the pattern $312$, which satisfies $p_1 = 3 = k$, we obtain
    \begin{align*}
        a_n^{(312, 231)} = a_n^{(231, 312)} = \sum_{\alpha = 1}^n \left(\sum_{\beta = 0}^{\alpha - 1} s_{\beta}^{(231, 312)}s_{n-\beta-1}^{(231, 12)}\right) = \sum_{j = 0}^{n-1} (n-j)s_j^{(231, 312)},
    \end{align*}
    where the second equality follows because $s_{k}^{(231, 12)} = 1$ for all $k$. (That is, the only permutation of length $k$ avoiding $12$ is $k \cdots 1$).

    It follows that
    \begin{equation*}
        f_n^{(312, 231)} = 1 + a_n^{(231, 312)} = 1+ \sum_{j = 0}^{n-1} (n-j)s_j^{(231, 312)} = 1 + n + \sum_{j = 1}^{n-1} (n-j)2^{j-1},
    \end{equation*}
    where the third equality follows from Lemma \ref{lem:pattern avoiding perm} together with $s_0^{(231, 312)} = 1$.

    We show the right-hand side is $2^n$ via a generating function argument. Let
    $$G(x) = \sum_{n \geq 0} \left(1 + n + \sum_{j = 1}^{n-1} (n-j)2^{j-1}\right)x^n,$$
    so that $f_n^{(312, 231)} = [x^n]G(x)$ for every $n \geq 1$. (We work with $G$ rather than with $\mathcal{F}^{(312, 231)}$ because the two series differ in their constant terms: $f_0^{(312, 231)} = 0$ by convention, whereas the summand above equals $1$ at $n = 0$.) We have
    \begin{align*}
        G(x) &= \sum_{n \geq 0} \left(1 + n + \sum_{j = 1}^{n-1} (n-j)2^{j-1} \right) x^n \\
        &= \left(\sum_{n \geq 0} x^n\right) + \left(\sum_{n \geq 0} nx^n\right)  + \left(\sum_{n \geq 0} \left(\sum_{j = 1}^{n-1}(n-j)2^{j-1}\right)x^n\right) \\
        &= \frac{1}{1-x} + \frac{x}{(1-x)^2} + \frac{x^2}{(1-x)^2(1-2x)} \\
        &= \frac{1}{1-2x} \\
        &= \sum_{n \geq 0} 2^nx^n.
    \end{align*}
    
    Hence, $f_n^{(312, 231)} = 2^n$ for every $n \geq 1$. 
\end{proof}

\begin{prop}
    The number of $(312, 213)$-avoiding TITOs is $f_n^{(312, 213)} = 2^n.$
\end{prop}
\begin{proof}
    We divide into cases according to Lemma \ref{lem:312 blocks}. 
    
    Let $\preceq$ be a $(312, 213)$-avoiding TITO. If $\preceq$ consists of exactly one waning block, then by Lemma \ref{lem:312 blocks} $\preceq$ must be $[\underline{n, \dots, 1}]$. If $\preceq$ consists of exactly one waxing block, then $\preceq$ must be $[1, \dots, n]$ because $\preceq$ must be $21$-avoiding, since an inversion $(a,b)$ in the window would give the $213$ pattern $b \preceq a \preceq a+n$.

    If $\preceq$ has two blocks, then the terms in the waxing block must be increasing to avoid $213$, and the resulting TITO remains $312$-avoiding as well. This case contributes $\sum_{i = 1}^{n-1} \binom{n}{i}$ because there are $\sum_{i = 1}^{n-1} \binom{n}{i}$ ways to partition the $n$ congruence classes into two nonempty blocks.
    
    Hence, $f_n^{(312, 213)} = \sum_{i = 0}^n \binom{n}{i} = 2^n$ for every $n$.
\end{proof}

\begin{prop}
    The number of $(312, 321)$-avoiding TITOs is $f_n^{(312, 321)} = 2^n - 1.$
\end{prop}
\begin{proof}
    A TITO avoiding $321$ cannot have a waning block. Thus, by Lemma \ref{lem:312 blocks}, any $(312, 321)$-avoiding TITO must consist of exactly one waxing block, so $f_n^{(312, 321)} = a_n^{(312, 321)} = a_n^{(231, 321)}$ by Lemma \ref{lem:inversion}. 
    By Lemma \ref{lem: formula affine avoid n---}, we have 
    \begin{align*}
        a_n^{(231, 321)} = \sum_{\alpha = 1}^n \left(\sum_{\beta = 0}^{\alpha - 1} s_{\beta}^{(231, 321)}s_{n-\beta-1}^{(231, 21)}\right) = \sum_{j = 0}^{n-1} (n-j)s_j^{(231, 321)}
    \end{align*}
    where the second equality follows because $s_{k}^{(231, 21)} = 1$ for all $k$. (That is, the only permutation of length $k$ avoiding $21$ is $1 \cdots k$).

    By Lemma \ref{lem:pattern avoiding perm}, $s_j^{(231, 321)} = s_j^{(231, 312)}$, so $a_n^{(312,321)} = a_n^{(231,321)} = \sum_{j=0}^{n-1}(n-j)s_j^{(231,312)} = a_n^{(312,231)}$. Recall from the proof of Proposition \ref{prop: 312 231} that this same sum equals $f_n^{(312,231)} - 1$, since $f_n^{(312,231)} = 1 + a_n^{(312,231)}$, and the $1$ accounts for the single waning-block TITO $[\underline{n,\dots,1}]$. Thus, we get
$$
f_n^{(312,321)} = a_n^{(312,321)} = a_n^{(312,231)} = f_n^{(312,231)} - 1 = 2^n - 1,
$$
where the $-1$ comes from excluding the case $[\underline{n,\dots,1}]$.
\end{proof}

\begin{prop}
    The number of $(123, 321)$-avoiding TITOs is $f_n^{(123, 321)} = 0.$
\end{prop}
\begin{proof}
    Any $(123, 321)$-avoiding TITO cannot have a waxing block or a waning block, so $f_n^{(123, 321)} = 0.$
\end{proof}

From the counting in the previous section, we obtain the following generating functions.

\begin{cor}
    For distinct $p, q \in S_3$, we have that $\mathcal{F}^{(p, q)}$ takes one of five forms:
    \begin{enumerate}[label=(\alph*)]
        \item $\mathcal{F}^{(312, 123)}(x) = \frac{x}{1-x}$; 
        \item $\mathcal{F}^{(312, 132)}(x) = \frac{2x}{1-x}$; 
        \item $\mathcal{F}^{(312, 231)}(x) = \mathcal{F}^{(312, 213)}(x) = \frac{2x}{1-2x}$; 
        \item $\mathcal{F}^{(312, 321)}(x) = \frac{x}{(1-2x)(1-x)}$;
        \item $\mathcal{F}^{(123, 321)}(x) = 0$.
    \end{enumerate}
\end{cor}

\section{TITOs avoiding 312 and $p \in S_4$}\label{sec: results length 4}

In this section, we enumerate $(p, q)$-avoiding TITOs for $p \in S_3 \setminus \{123, 321\}$ and $q \in S_4$. We use antiautomorphisms $\Psi_{\leftrightarrow}$ and $\Psi_{\updownarrow}$ to reduce to the case of $(312, p)$-avoidance for $p \in S_4$. If $p$ contains the pattern $312$, then $f_n^{(312, p)} = f_n^{(312)}=\binom{2n}{n}$ by Barkley and Defant \cite{BD25+}.

The remaining $312$-avoiding length $4$ patterns fall into equivalence classes. Note that no nontrivial composition of $\Psi_\leftrightarrow$ and $\Psi_\updownarrow$ fixes the pattern $312$, so these classes do not arise from the symmetries used above; instead, they emerge from the classification of $(312,p)$-avoiding affine permutations developed in \S\ref{sec: affine count} and from the arc diagram analysis in \S\ref{subsubsec: arc diagram approach}.
\begin{enumerate}[label=(\alph*)]
\item $f_n^{(312, 1234)}$;
\item $f_n^{(312, 1243)}$;
\item $f_n^{(312, 4321)}$;
\item $f_n^{(312, 3421)} = f_n^{(312, 2431)}$;
\item $f_n^{(312, 3241)} = f_n^{(312, 3214)} = f_n^{(312, 2314)}$;
\item $f_n^{(312, 2341)}$;
\item $f_n^{(312, 2134)} = f_n^{(312, 2143)} = f_n^{(312, 1324)} = f_n^{(312, 1432)} = f_n^{(312, 1342)}$.
\end{enumerate}
We treat each case above separately to obtain the following results for the generating functions $\mathcal{F}^{(312,p)}(x)$ and the sequences $f_n^{(312,p)}$.

\begin{thm}\label{thm: 312 p avoiding TITO summary}
For $n \geq 1$, the cases above have the following enumerations:
\begin{enumerate}[label=(\alph*)]\item $\mathcal{F}^{(312, 1234)}(x) = \frac{x}{1-x}$ and $f_n^{(312, 1234)} = 1$;\item $\mathcal{F}^{(312, 1243)}(x) = \frac{2x}{1-x}$ and $f_n^{(312, 1243)} = 2$;\item $\mathcal{F}^{(312, 4321)}(x) = \frac{x(2x^2 - 2x + 1)}{(1-2x)(1-3x+x^2)}$ and$$f_n^{(312, 4321)} = \left(\frac{3+\sqrt{5}}{2}\right)^n + \left(\frac{3-\sqrt{5}}{2}\right)^n - 2^n;$$\item $\mathcal{F}^{(312, p)}(x) = \frac{x(2 - 8x + 11x^2 - 4x^3)}{(1-x)(1-2x)(1-3x+x^2)}$ and$$f_n^{(312, p)} = \left(\frac{3+\sqrt{5}}{2}\right)^n + \left(\frac{3-\sqrt{5}}{2}\right)^n - 2^n + 1$$for $p \in \{3421, 2431\}$;\item $\mathcal{F}^{(312, p)}(x) = \frac{-x(x^2-2x+2)}{(x-1)(x^2-3x+1)}$ and$$f_n^{(312, p)} = \left(\frac{3+\sqrt{5}}{2}\right)^n + \left(\frac{3-\sqrt{5}}{2}\right)^n - 1$$for $p \in \{3241, 3214, 2314\}$;\item $\mathcal{F}^{(312, 2341)}(x) = \frac{x(3x^3-9x^2+6x-2)}{(1-x)(3x^3 - 5x^2 + 4x - 1)}$ and$$f_n^{(312, 2341)} = -2 + \lambda_1^n + \lambda_2^n + \lambda_3^n,$$where $\lambda_1, \lambda_2, \lambda_3$ are the reciprocals of the roots of $3x^3 - 5x^2 + 4x - 1$;\item $\mathcal{F}^{(312, p)}(x) = \frac{2x}{1-2x}$ and $f_n^{(312, p)} = 2^n$ for $p \in \{2134, 2143, 1324, 1432, 1342\}$.\end{enumerate}\end{thm}

We have two main techniques for counting $(312, p)$-avoiding TITOs, and we divide this section into subsections accordingly. First, in $\S \ref{subsubsec: arc diagram approach}$, we use the bijection between $312$-avoiding TITOs and noncrossing arc diagrams to provide enumerative formulas for $f_n^{(312, p)}$ for $p \in \{3214, 2314\}$. Second, in $\S \ref{subsubsec: direct approach}$, we use affine permutations to compute $f_n^{(312, p)}$ for all $p \in S_4 \setminus \{3214, 2314\}$. 

\subsection{Enumerating $(312, p)$-avoiding affine permutations}\label{sec: affine count}
Recall from Lemma \ref{lem:inversion} that $a_n^{(312, p)} = a_n^{(231, p^{-1})}$ for all $p$. Using the characterization of the $(231, p)$-avoiding affine permutations given in $\S \ref{sec: affine permutations}$, we arrive at the following theorem.

\begin{thm}\label{thm: 312 p avoiding affine}
    For patterns $p$ of length $4$ avoiding $312$ and for $n \geq 1$, we have 
    \begin{align*}
        a_n^{(312, p)} &= \begin{cases}
            0 &\text{if } p = 1234;\\
            1 &\text{if }p \in \{1243, 2134, 2143, 1324\}; \\
            2^n - 1 &\text{if } p \in \{1432, 1342, 3214, 2314\}; \\
            \left(\frac{3+\sqrt{5}}{2}\right)^n + \left(\frac{3-\sqrt{5}}{2}\right)^n - 2^n &\text{if } p \in \{4321, 3421, 2431, 3241\};\\
            \lambda_1^n + \lambda_2^n + \lambda_3^n - 3 &\text{if } p = 2341,
        \end{cases}
    \end{align*}
    where $\alpha \in \R$ and $\beta, \overline{\beta} \in \C$ are the roots of $3x^3 - 5x^2 + 4x - 1$, and $\lambda_1 = 1/\alpha$, $\lambda_2 = 1/\beta$, and $\lambda_3 = 1/\overline{\beta}$.
\end{thm}

We consider each case separately below.
\begin{prop}
    The number of $(312, 1234)$-avoiding affine permutations is $a_n^{(312, 1234)} = 0$.
\end{prop}
\begin{proof}
    No affine permutation avoids $1234$ because, for any integer $x$, the elements $$x \preceq x + n \preceq x + 2n \preceq x + 3n$$ form a 1234 pattern. Therefore, $a_n^{(312, 1234)} = 0$.
\end{proof}

\begin{prop}
    For $p \in \{1243, 2134, 2143, 1324\}$, the number of $(312, p)$-avoiding affine permutations is $a_n^{(312, p)} = 1$.
\end{prop}

\begin{proof}
    For $p \in \{1243, 2134, 2143, 1324\}$, we show that the only 312-avoiding affine permutation in $\tilde{S}_n$ that also avoids $p$ is $[1, 2, \dots, n]$. Indeed, if the elements of the waxing block do not appear in increasing order, then some window of that block contains a lower wall $(a,b)$. Moreover, since $\preceq$ is $312$-avoiding, Lemma \ref{lem:windows of 312} gives $b - a < n$, i.e., $a < b < a + n$. However, this implies that
    \begin{align*}
        b-2n \preceq b - n \preceq b \preceq a \quad \text{is a 1243 pattern};\\
        b \preceq a \preceq a+n \preceq a+2n \quad \text{is a $2134$ pattern}; \\
        b \preceq a \preceq b+n \preceq a+n \quad \text{is a $2143$ pattern}; \\
        a-n \preceq b \preceq a \preceq b+n \quad \text{is a $1324$ pattern}.
    \end{align*}

    \noindent Thus, $a_n^{(312, p)} = 1$ for $p \in \{1243, 2134, 2143, 1324\}$. 
\end{proof}

\begin{prop}
    For $p \in \{1432, 1342, 3214, 2314\}$, the number of $(312, p)$-avoiding affine permutations is $a_n^{(312, p)} = 2^n - 1$.
\end{prop}
\begin{proof}
    We consider $p \in \{1432, 1342, 3214, 2314\}$, so $p^{-1} \in \{1432, 1423, 3214, 3124\}$. Applying Lemmas \ref{lem: formula affine avoid 1n--} and \ref{lem: formula affine avoid ---n} (note that $3214$ and $3124$ satisfy the hypotheses $p_1 = k-1$ and $p_k = k$ of the latter), we have that 
    \begin{align*}
        a_n^{(312, 1432)} = a_n^{(231, 1432)} = \sum_{j = 0}^{n-1} (n-j) s_{j}^{(231, 321)}s_{n-j-1}^{(231, 21)} = \sum_{j = 0}^{n-1} (n-j) s_{j}^{(231, 321)}; \\
        a_n^{(312, 1342)} = a_n^{(231, 1423)} = \sum_{j = 0}^{n-1} (n-j) s_{j}^{(231, 312)}s_{n-j-1}^{(231, 12)} = \sum_{j = 0}^{n-1} (n-j) s_{j}^{(231, 312)}; \\
        a_n^{(312, 3214)} = a_n^{(231, 3214)} = \sum_{j = 0}^{n-1} (n-j) s_{j}^{(231, 321)}s_{n-j-1}^{(231, 21)} = \sum_{j = 0}^{n-1} (n-j) s_{j}^{(231, 321)}; \\
        a_n^{(312, 2314)} = a_n^{(231, 3124)} = \sum_{j = 0}^{n-1} (n-j) s_{j}^{(231, 312)}s_{n-j-1}^{(231, 12)} = \sum_{j = 0}^{n-1} (n-j) s_{j}^{(231, 312)}.
    \end{align*}
    By Lemma \ref{lem:pattern avoiding perm}, $s_{j}^{(231, 312)} = s_{j}^{(231, 321)}$, so it suffices to compute $\sum_{j = 0}^{n-1} (n-j)s_j^{(231, 321)}$. The proof of Proposition \ref{prop: 312 231} shows that this sum equals $2^n - 1$, so $$a_n^{(312, 1432)} = a_n^{(312, 1342)} = a_n^{(312, 3214)} = a_n^{(312, 2314)} = 2^n - 1.$$
\end{proof}
For $p \in \{4321, 3421, 2431, 3241\}$, we obtain a closed form for $a_n^{(312, p)}$ by first computing its generating function $\mathcal{A}^{(312, p)}(x)$.
\begin{lem}\label{lem: affine for 4321}
    For $p \in \{4321, 3421, 2431, 3241\}$, the number of $(312, p)$-avoiding affine permutations is $$a_n^{(312, p)} = \sum_{j = 0}^{n-1} (n-j)s_{j}^{(231, 4321)}s_{n-j-1}^{(231, 321)}.$$

    Moreover, the generating function is $$\mathcal{A}^{(312, p)}(x) = \sum_{n \geq 0}a_n^{(312, p)}x^n = \frac{x(2x^{2}-2x+1)}{(1-2x)(1-3x+x^{2})}.$$
\end{lem}

\begin{proof}
We consider $p \in \{4321, 3421, 2431, 3241\}$, so $p^{-1} \in \{4321, 4312, 4132, 4213\}$. Each of these permutations begins with $4$, with respective length-$3$ tails $321, 312, 132, 213$; by Lemma \ref{lem:pattern avoiding perm}, these tails all satisfy $s_n^{(231,\cdot)} = 2^{n-1}$ for $n \geq 1$. By Lemma \ref{lem: same count perm}, it follows that $s_n^{(231,4321)} = s_n^{(231,4312)} = s_n^{(231,4132)} = s_n^{(231,4213)}$ for every $n$, so we may uniformly write $s_j^{(231,p^{-1})} = s_j^{(231,4321)}$ regardless of which $p$ we consider. By the same reasoning applied to the length-$3$ tails, $s_{n-j-1}^{(231,p^{-1}_{[2:4]})} = s_{n-j-1}^{(231,321)}$ as well.

Applying Lemma \ref{lem: formula affine avoid n---} to $p^{-1}$, we therefore have
$$
    a_n^{(312, p)} = a_n^{(231,p^{-1})} = \sum_{j=0}^{n-1} (n-j) s_j^{(231,4321)} s_{n-j-1}^{(231,321)}
$$
for all $p \in \{4321, 3421, 2431, 3241\}$.

By Lemmas \ref{lem: generating function perm} and \ref{lem: perm generating functions}, we have the generating functions
$$
    \mathcal{S}^{(231,321)}(x) = \frac{1-x}{1-2x}, \qquad \mathcal{S}^{(231,4321)}(x) = \frac{1-2x}{1-3x+x^2}.
$$
Differentiating $\mathcal{S}^{(231,321)}(x)$ gives $\big(\mathcal{S}^{(231,321)}\big)'(x) = \dfrac{1}{(1-2x)^2}$, so that
$$
    \sum_{k \geq 0} (k+1) s_k^{(231,321)} x^k = \mathcal{S}^{(231,321)}(x) + x\big(\mathcal{S}^{(231,321)}\big)'(x).
$$
Multiplying by $x\mathcal{S}^{(231,4321)}(x)$ produces exactly the generating function of the desired sum, so
\begin{align*}
\mathcal{A}^{(312, p)}(x)
     &= x\mathcal{S}^{(231, 4321)}\mathcal{S}^{(231, 321)}(x)
        + x^{2}\mathcal{S}^{(231, 4321)}\big(\mathcal{S}^{(231, 321)}\big)'(x) \\[4pt]
     &= \frac{x(1-2x)}{1-3x+x^{2}}\cdot\frac{1-x}{1-2x}
        + x^{2}\cdot\frac{1-2x}{1-3x+x^{2}}\cdot\frac{1}{(1-2x)^{2}}\\[6pt]
     &= \frac{x(2x^{2}-2x+1)}{(1-2x)(1-3x+x^{2})},
\end{align*}
as desired.
\end{proof}

Extracting a closed form for the coefficients of a rational function is routine, so we simply state the result below.
\begin{prop}\label{prop: closed form 4321}
    For $p \in \{4321, 3421, 2431, 3241\}$ and $n \geq 1$, the number of $(312, p)$-avoiding affine permutations is $$a_n^{(312, p)}  = \left(\frac{3+\sqrt{5}}{2}\right)^n + \left(\frac{3-\sqrt{5}}{2}\right)^n - 2^n.$$
\end{prop}

To compute $a_n^{(312, 2341)}$, we follow a similar procedure. 

\begin{lem}\label{lem: affine for 2341}The number of $(312, 2341)$-avoiding affine permutations is$$a_n^{(312, 2341)} = \sum_{j = 0}^{n-1} (n-j)s_{j}^{(231, 4123)}s_{n-j-1}^{(231, 123)}.$$ Moreover, the generating function is 
$$\mathcal{A}^{(312, 2341)}(x) = \sum_{n \geq 0}a_n^{(312, 2341)}x^n = \frac{x(4x^2 -2x + 1)}{(1-x)(1-4x+5x^2-3x^3)}.$$
\end{lem}
\begin{proof}
    If $p = 2341$, then $p^{-1} = 4123$. By Lemma \ref{lem: formula affine avoid n---}, we have that
    \begin{align*}
        a_n^{(312, 2341)} = a_n^{(231, 4123)}= \sum_{j = 0}^{n-1} (n-j)s_{j}^{(231, 4123)}s_{n-j-1}^{(231, 123)}.
    \end{align*}

    \noindent By Lemmas \ref{lem: generating function perm} and \ref{lem: perm generating functions}, we obtain the generating functions $$\mathcal{S}^{(231, 123)}(x) = \frac{1-2x+2x^2}{(1-x)^3}, \quad \mathcal{S}^{(231, 4123)}(x) = \frac{(1-x)^3}{1-4x+5x^2-3x^3}.$$ 
    
    \noindent Since $\sum_{k \geq 0} (k+1)s_k^{(231, 123)}x^k = \mathcal{S}^{(231, 123)}(x) + x\big(\mathcal{S}^{(231, 123)}\big)'(x)$, we get
\begin{align*}
\mathcal{A}^{(312, 2341)}(x)
     &= \sum_{n\ge 0}a_n^{(231, 4123)}x^{n} \\[4pt]
     &= x \cdot \mathcal{S}^{(231, 4123)}(x)\left(\mathcal{S}^{(231, 123)}(x)+ x\mathcal{S}^{(231, 123)'}(x) \right) \\
     &= x\cdot \frac{(1-x)^{3}}{1-4x+5x^{2}-3x^{3}}\cdot
  \frac{4x^{2}-2x+1}{(1-x)^{4}}
  \\
  &=
  \frac{x(4x^2 -2x + 1)}{(1-x)(1-4x+5x^2-3x^3)},
\end{align*}
as desired.
\end{proof}

The rational generating function yields the following closed formula.

\begin{prop}\label{prop: closed form 2341} Let $\alpha \in \mathbb{R}$ and $\beta, \overline{\beta} \in \C$ be the real and complex-conjugate roots of $3x^3 - 5x^2 + 4x - 1$. Set $\lambda_{1}=1/\alpha, \lambda_{2}=1/\beta, \lambda_{3}=1/\bar\beta$. The number of $(312, 2341)$-avoiding affine permutations is $a_n^{(312, 2341)}  = \lambda_1^n + \lambda_2^n + \lambda_3^n - 3$.\end{prop}

It is possible to calculate explicit formulas for the $\lambda_i$s. However, the exact forms are not insightful and are therefore omitted here.

\subsection{Enumerating $(312, p)$-avoiding TITOs}\label{sec: titos count}

In this section, we enumerate $(312, p)$-avoiding TITOs for $p \in S_4$ using Theorem \ref{thm: 312 p avoiding affine} and Propositions \ref{prop: closed form 4321} and \ref{prop: closed form 2341}.

\begin{thm}\label{thm: TITO enumeration length 4}
    For $p \in S_4$ and $n \geq 1$, we have that
    \begin{enumerate}[label=(\alph*)]
        \item $f_n^{(312, p)} = \binom{2n}{n}$ if $p$ contains $312$;
        \item $f_n^{(312, 1234)} = 1$; 
        \item $f_n^{(312, 1243)} = 2$;
        \item $f_n^{(312, 4321)} = a_n^{(312, 4321)} = \left(\frac{3+\sqrt{5}}{2}\right)^n + \left(\frac{3-\sqrt{5}}{2}\right)^n - 2^n$; 
        \item $f_n^{(312, 2341)} =  a_n^{(312, 2341)} + 1 = -2 + \lambda_1^{n}+\lambda_2^{n}+\lambda_3^{n}$;
        \item $f_n^{(312, 3421)} = f_n^{(312, 2431)} = a_{n}^{(312, 4321)} + 1 = \left(\frac{3+\sqrt{5}}{2}\right)^n + \left(\frac{3-\sqrt{5}}{2}\right)^n - 2^n + 1$;
        \item $f_n^{(312, 3241)} = f_n^{(312, 3214)} = f_n^{(312, 2314)} =  f_n^{(312, 4321)} + 2^n - 1 = \left(\frac{3+\sqrt{5}}{2}\right)^n + \left(\frac{3-\sqrt{5}}{2}\right)^n - 1$;
        \item $f_n^{(312, 2134)} = f_n^{(312, 2143)} = f_n^{(312, 1324)} = f_n^{(312, 1432)} = f_n^{(312, 1342)} = 2^n$;
    \end{enumerate}
    where $\alpha \in \R$ and $\beta, \overline{\beta} \in \C$ are the roots of $3x^3 - 5x^2 + 4x - 1$, and $\lambda_{1}=1/\alpha, \lambda_{2}=1/\beta, \lambda_{3}=1/\overline{\beta}$.
\end{thm}

We treat these cases separately, organizing them by the techniques used in their proofs.

\subsubsection{Arc diagram approach}\label{subsubsec: arc diagram approach}

We use arc diagrams to prove $$f_n^{(312, 3214)} = f_n^{(312, 2314)} = a_n^{(312, 4321)} + 2^n - 1.$$ Recall that $b_n^{(p,q)}=f_n^{(p,q)}-a_n^{(p,q)}$ counts projective $(p,q)$-avoiding TITOs, where a TITO is projective if it is not an affine permutation. We therefore only need to count the projective TITOs.

From Theorem \ref{thm: 312 p avoiding affine}, we know that $a_n^{(312, 3214)} = a_n^{(312, 2314)} = 2^n - 1$, so it suffices to show $$b_n^{(312, 3214)} = b_n^{(312, 2314)} = a_n^{(312, 4321)}.$$

Observe that a $312$-avoiding TITO is projective if and only if it has a waning block: by Lemma \ref{lem:312 blocks}, such a TITO consists of at most two blocks, the first waxing and the second waning, and it is an affine permutation exactly when it consists of a single waxing block.

We now translate the $3214$- and $2314$-avoidance conditions into the language of arc diagrams and enumerate the corresponding diagrams. We say an arc $\gamma_{a, b}$ is \dfn{nested} in $\gamma_{c, d}$ if $\gamma_{a, b}$ lies in the region bounded by $\gamma_{c, d}$ and the portion of the outer boundary running clockwise from $v_c$ to $v_d$; equivalently, $\gamma_{a,b}$ stays on our left as we walk along $\gamma_{c,d}$. 

\begin{lem}\label{lem:3214 arc}
    The number of projective $(312, 3214)$-avoiding TITOs $b_n^{(312, 3214)}$ is equal to the number of noncrossing arc diagrams $A$ satisfying: 
    \begin{enumerate}
        \item $A$ has a nonempty cycle $C$;
        \item every walker $W_i \in \mathsf{W}(A)$ contains no nested arcs;
        \item every walker $W_i \in \mathsf{W}(A)$ contains no paths with more than one step.
    \end{enumerate}
\end{lem}
\begin{proof}
    By Lemmas \ref{lem:312 blocks} and \ref{lem:cycle}, the arc diagram $\mathcal{A}(\preceq)$ of any projective $312$-avoiding TITO must include exactly one cycle that is uniquely determined by its waning block. Moreover, the waning block must be strictly decreasing. Thus, a projective $312$-avoiding TITO is also $3214$-avoiding if and only if the waxing block of $\preceq$ avoids $321$. It therefore suffices to show that this waxing block avoids $321$ if and only if the walkers in $\mathsf{W}(\mathcal{A}(\preceq))$ have neither nested arcs nor paths with more than one step. 
    
    We first consider the forward direction, and claim by induction on $|E(W)|$ that if a walker $W$ has no nested arcs and no path with more than one step, then the subarray $\mathcal{W}$ produced by Algorithm~\ref{algo:support} avoids $321$. The case $E(W) = \varnothing$ is clear, as $\mathcal{W}$ is then increasing. Otherwise the subcompass path is a single arc, say from $v_a$ to $v_b$, so Algorithm~\ref{algo:support} returns
    $$
      \mathcal{W} = D_0 \circ D_1 \circ [\phi(v_b), \phi(v_a)] \circ E,
    $$
    where $D_0$, $D_1$ and $E$ are the subarrays arising from the subwalkers whose compass labels are smaller than $\phi(v_a)$, lie strictly between $\phi(v_a)$ and $\phi(v_b)$, and are larger than $\phi(v_b)$, respectively. An arc in the second of these subwalkers would be nested in $\gamma_{a,b}$, and an arc in the first would have a source with compass label smaller than $\phi(v_a)$, contradicting the minimality in the choice of $v_a$. Thus, $D_0$ and $D_1$ are increasing. The third subwalker inherits both hypotheses and has fewer arcs, so $E$ avoids $321$ by induction. Every entry of $D_0$ is smaller than $\phi(v_a)$, every entry of $D_1$ lies between $\phi(v_a)$ and $\phi(v_b)$, and every entry of $E$ is larger than $\phi(v_b)$, so any occurrence of $321$ in $\mathcal{W}$ lies in $E$, and there is none. 
    
    Since $A$ has a cycle, Algorithm~\ref{algo: inverse construction} returns $\mathcal{W}_{\mathrm{wax}} = \mathcal{W}_{i_1} \circ \cdots \circ  \mathcal{W}_{i_r}$, with the walkers in increasing order of subscript. The compass labels of distinct walkers form disjoint integer intervals ordered by their subscripts, so an occurrence of $321$ would lie in a single $\mathcal{W}_{i_s}$. Thus, by our earlier argument, $\mathcal{W}_{\mathrm{wax}}$ avoids $321$. Finally, since the compass labels are $n$ consecutive integers, every entry of $\mathcal{W}_{\mathrm{wax}}$ is smaller than every entry of $\mathcal{W}_{\mathrm{wax}}$ increased by $n$. An occurrence of $321$ appears in $\preceq$ as a decreasing sequence of integers, while the translates of $\mathcal{W}_{\mathrm{wax}}$ occur in $\preceq$ in increasing order of their entries, so all three of its elements lie in a single translate of $\mathcal{W}_{\mathrm{wax}}$.
    Hence, the waxing block avoids $321$.

    We now show the converse. Let $\phi$ be the compass labeling of $A$. If $W_i$ has an arc $\gamma_{a,b}$ nested in an arc $\gamma_{c, d}$, then $\phi(v_c) < \phi(v_a) < \phi(v_b) < \phi(v_d)$. Moreover, $\phi(v_b) \preceq \phi(v_a) \preceq \phi(v_d) \preceq \phi(v_c)$, so $\phi(v_b) ,\phi(v_a), \phi(v_c)$ form a $321$ pattern. If $W_i$ has a path with more than one step, say $\gamma_{a, b}$ and $\gamma_{b,c}$, then $\phi(v_a) < \phi(v_b) < \phi(v_c)$ but $\phi(v_c) \preceq \phi(v_b) \preceq \phi(v_a)$, again forming a $321$ pattern.
\end{proof}

We now characterize the $2314$ pattern avoidance constraint in terms of arc diagrams.
\begin{lem}\label{lem:2314 arc}
    The number of projective $(312, 2314)$-avoiding TITOs $b_n^{(312, 2314)}$ is equal to the number of noncrossing arc diagrams $A$ satisfying:
    \begin{enumerate}
        \item $A$ has a nonempty cycle $C$; 
        \item every walker $W_i \in \mathsf{W}(A)$ contains no arc with length greater than one.
    \end{enumerate}
\end{lem}
\begin{proof}
    By the same reductions as in the proof of Lemma \ref{lem:3214 arc}, it suffices to show that this waxing block avoids $231$ if and only if the walkers in $\mathsf{W}(\mathcal{A}(\preceq))$ have no arcs of length greater than one. 
    
    We first consider the forward direction. If the walkers contain no arc with length greater than one, then by Algorithm \ref{algo: inverse construction}, no subwalker can contain a path with arcs of length greater than one. Thus, any path within a subwalker (if one exists) must consist entirely of length-$1$ arcs, giving rise to a subarray of strictly decreasing consecutive integers. Such a subarray avoids $231$, so the resulting waxing window subarray avoids $231$ as well. Moreover, distinct walkers have disjoint compass-label intervals, ordered by their subscripts, so concatenating the subarrays $\mathcal{W}_i$ in increasing order of subscript cannot create an occurrence of $231$ using two different walkers; such an occurrence would need its first entry in $\preceq$ to exceed its last, whereas every entry of $\mathcal{W}_a$ is smaller than every entry of $\mathcal{W}_b$ when $a < b$. Hence $\mathcal{W}_{\mathrm{wax}}$ avoids $231$.

Since the compass labels are $n$ consecutive integers, every entry of $\mathcal{W}_{\mathrm{wax}}$ is smaller than every entry of $\mathcal{W}_{\mathrm{wax}}$ increased by $n$. An occurrence of $231$ appears in $\preceq$ as a triple $i_2 \preceq i_3 \preceq i_1$ with $i_1 < i_2 < i_3$, whose last entry is smaller than its first. If $i_2$ and $i_1$ were in different translates, the translate containing $i_2$ would precede the one containing $i_1$ and we would get $i_2 < i_1$, which is a contradiction. Thus, all three entries lie in a single translate of $\mathcal{W}_{\mathrm{wax}}$, and the waxing block avoids $231$.

    Conversely, consider a walker containing an arc $\gamma_{a,b}$ with length greater than one. Let $\phi$ be the compass labeling of $A$. Then, the subwalker under $\gamma_{a, b}$ is nonempty and includes at least one boundary point, which we call $v_c$. Then $\phi(v_a) < \phi(v_c) < \phi(v_b)$, whereas $\phi(v_c) \preceq \phi(v_b) \preceq \phi(v_a)$, which is a 231-pattern.
\end{proof}

Next, we enumerate the walkers corresponding to projective $(312,p)$-avoiding TITOs for $p \in \{3214, 2314\}$. For convenience, we introduce the following notation:
\begin{align*}
    w_k &:= \# \{\text{walkers $W$ with $|V(W)| = k$, no nested arcs or path with more than one step}\}; \\
    w_k' &:=\# \{\text{walkers $W$ with $|V(W)| = k$ and no arcs with length greater than one}\}.
\end{align*}

Let $w_0 = w_0' = 1$.

\begin{lem}\label{lem: walker count}
    We have that $w_k = w_k' = 2^{k-1}$ for $k \geq 1$. 
\end{lem}
\begin{proof}
    Fix a walker $W$. Denote the boundary point in $W$ with the smallest compass label by $v_i$. We condition on what happens at $v_i$ because $v_i$ has no incoming arc.

    We first consider $w_k$. If $v_i$ has no outgoing arc, then deleting $v_i$ from $W$ gives a bijection with walkers on $k-1$ boundary points with no nested arcs and no path of more than one step, so this case contributes $w_{k-1}$ walkers. If instead $v_i$ has an outgoing arc of length $j$, where $1 \leq j \leq k-1$, then the points enclosed by this arc must form a subwalker with no arcs of its own, since any such arc would either be nested inside the arc from $v_i$ or extend the path starting at $v_i$ to more than one step. Hence the enclosed subwalker is uniquely determined by $j$. Deleting this arc, both of its endpoints, and the $j-1$ boundary points strictly enclosed by it gives a bijection with walkers on $k-j-1$ boundary points satisfying the same restrictions. Thus, this case contributes $w_{k-j-1}$ walkers. Summing over all cases, we have that
    $$w_k = \sum_{i = 0}^{k-1} w_i,$$
    with $w_0 = 1$. This recursive relation implies that $w_k = 2^{k-1}$ for $k \geq 1$. 

    Now, we consider $w_k'$. We again have two cases: $v_i$ has either no outgoing arc or is in a path with $j$ steps, where $1 \leq j \leq k-1$ and every step has length $1$. In the first case, deleting $v_i$ gives a bijection with the walkers counted by $w'_{k-1}$. In the second case, since every step has length $1$, the subwalkers enclosed by the path must be empty, so deleting $v_i$ and its path gives a bijection with the walkers counted by $w'_{k-j-1}$. Thus, we obtain the same recurrence relation
    \begin{align*}
        w_k' = \sum_{i = 0}^{k-1} w_i',
    \end{align*}
    with $w_0' = 1$. This implies $w_k' = 2^{k-1}$ for $k \geq 1$. 
\end{proof}

We introduce a few more definitions for our enumeration. A \dfn{weak composition} of a nonnegative integer $n$ is a way of writing $n$ as an ordered sum of nonnegative integers, where the order of the summands (also known as \dfn{parts} of the composition) matters. Let $\mathscr{C}_{n}^k := \{\text{weak compositions of $n$ with exactly $k$ parts}\}$.

\begin{thm}
    The number of projective $(312, p)$-avoiding TITOs for $p \in \{3214, 2314\}$ is
    \begin{align*}
        b_n^{(312, 3214)} = b_n^{(312, 2314)} = \sum_{k = 1}^{n} \left(\sum_{\lambda \in \mathscr{C}_{n-k}^{k}} (1+\lambda_1) w_{\lambda}\right),
    \end{align*}
    where $w_{\lambda} = w_{\lambda_1} \cdots w_{\lambda_k}.$

    Moreover, we get $$b_n^{(312, 3214)} = b_n^{(312, 2314)} = a_n^{(312, 4321)} = \left(\frac{3+\sqrt{5}}{2}\right)^n + \left(\frac{3-\sqrt{5}}{2}\right)^n - 2^n.$$
\end{thm}

\begin{proof}
    By Lemmas \ref{lem:3214 arc} and \ref{lem:2314 arc}, both $b_n^{(312,3214)}$ and $b_n^{(312,2314)}$ count noncrossing arc diagrams with a nonempty cycle $C$ in which every walker satisfies the stated arc restriction; by Lemma \ref{lem: walker count}, there are $w_m = w_m' = 2^{m-1}$ such walkers on $m$ boundary points for each $m \geq 1$. For such a diagram containing a $k$-cycle, reading off the number of boundary points in each walker, in clockwise order beginning with the gap containing $v_1$ (or, if $v_1$ lies on the cycle, with the gap immediately following $v_1$), gives a weak composition of $n-k$ into $k$ parts; conversely, any choice of $k$-cycle position together with a weak composition of $n-k$ into $k$ parts and a choice of walker of the appropriate size in each part determines such a diagram uniquely. For instance, in Example \ref{ex: compass labeling}, we obtain the weak composition $(5,4,0)$ of $9$.

    To obtain $b_n^{(312, 3214)} = b_n^{(312, 2314)}$, we distinguish two cases according to the position of $v_1$ for both $b_n^{(312, 3214)}$ and $b_n^{(312, 2314)}$:
    the boundary point $v_1$ is either in the cycle or in a walker. If $v_1$ lies in a $k$-vertex cycle, then by the correspondence above, the diagrams in this case are in bijection with pairs consisting of a weak composition $\lambda\in\mathscr C_{n-k}^k$ and, for each part $\lambda_j$, a choice of one of the $w_{\lambda_j}$ walkers on $\lambda_j$ boundary points; since $v_1$ itself lies in the cycle and contributes no further choice, this case contributes
    $$\sum_{k = 1}^{n} \left(\sum_{\lambda \in \mathscr{C}_{n-k}^{k}} w_{\lambda}\right)$$
    diagrams, for both $b_n^{(312, 3214)}$ and $b_n^{(312, 2314)}$.

    If instead $v_1$ lies in a walker, then we again condition on the size $k$ of the cycle and the resulting weak composition $\lambda \in \mathscr{C}_{n-k}^k$ of walker sizes; here $v_1$ may additionally occupy any of the $\lambda_1$ boundary points of the first walker in clockwise order, so this case contributes
    $$\sum_{k = 1}^{n} \left(\sum_{\lambda \in \mathscr{C}_{n-k}^{k}} \lambda_1 w_{\lambda}\right)$$
    diagrams, for both $b_n^{(312, 3214)}$ and $b_n^{(312, 2314)}$. Thus, we obtain the equality $b_n^{(312, 3214)} = b_n^{(312, 2314)}$, and we now compute their value using generating functions.

    For simplicity, we use $\mathcal{B}(x)$ to denote $\mathcal{B}^{(312, 3214)}(x) = \mathcal{B}^{(312, 2314)}(x)$ in this proof. To show $b_n^{(312, 3214)} = b_n^{(312, 2314)} = a_n^{(312, 4321)}$, we prove that $$\mathcal{B}(x) = \mathcal{A}^{(312, 4321)}(x).$$
    
    By Lemma \ref{lem: walker count}, the generating function for the sequence $(w_n)_{n\geq0}$ is $\mathcal{W}(x) = \frac{1-x}{1-2x}$ because $w_0 = 1$ and $w_k = w_k' = 2^{k-1}$ for $k \geq 1$. Let $S_{n,k}$ denote the inner sum in the expression for $b_n^{(312,3214)}$ given in the theorem statement, so
$$
S_{n,k} = \sum_{\lambda \in \mathscr{C}_{n-k}^{k}} (1+\lambda_1) w_{\lambda} = \sum_{\lambda_1+\dots+\lambda_k=n-k} (1+\lambda_1) w_{\lambda_1} \cdots w_{\lambda_k}.
$$
This sum can be split into two sums
$$
S_{n,k} = \sum_{\lambda_1+\dots+\lambda_k=n-k} w_{\lambda_1} \cdots w_{\lambda_k} + \sum_{\lambda_1+\dots+\lambda_k=n-k} \lambda_1 w_{\lambda_1}  \cdots w_{\lambda_k},
$$
where the first sum is the coefficient of $x^{n-k}$ in the expansion of $(\mathcal{W}(x))^k$, i.e.
$$
\sum_{\lambda_1+\dots+\lambda_k=n-k} w_{\lambda_1} \cdots w_{\lambda_k} = [x^{n-k}] (\mathcal{W}(x))^k.
$$
For the second sum, we define a new generating function $D(x)$ for the sequence $\{k w_k\}_{k \ge 0}$ via
\begin{align*}
D(x) &= \sum_{k=0}^{\infty} k w_k x^k = x \frac{d}{dx} \mathcal{W}(x) = \frac{x}{(1-2x)^2}.
\end{align*}
Then, the second sum is the coefficient of $x^{n-k}$ in $D(x)(\mathcal{W}(x))^{k-1}$, i.e.
$$
\sum_{\lambda_1+\dots+\lambda_k=n-k} \lambda_1 w_{\lambda_1} w_{\lambda_2} \cdots w_{\lambda_k} = [x^{n-k}] (D(x) (\mathcal{W}(x))^{k-1}).
$$
Combining these, we see $S_{n,k}$ is the coefficient of $x^{n-k}$ in the sum of these two generating functions
$$
S_{n,k} = [x^{n-k}] (\mathcal{W}(x)^k + D(x)\mathcal{W}(x)^{k-1}) = [x^{n-k}] \left( (\mathcal{W}(x)+D(x))\mathcal{W}(x)^{k-1} \right).
$$

Because $b_n = \sum_{k=1}^n S_{n,k}$, we can now express the generating function $\mathcal{B}(x)$ in terms of $\mathcal{W}(x)$ and $D(x)$ as
$$
\mathcal{B}(x) = \sum_{n=1}^{\infty} b_n x^n = \sum_{n=1}^{\infty} \left( \sum_{k=1}^{n} [z^{n-k}] \left( (\mathcal{W}(z)+D(z))\mathcal{W}(z)^{k-1} \right) \right) x^n.
$$
Let $G(z) = \mathcal{W}(z)+D(z)$. Changing the order of summation gives
$$
\mathcal{B}(x) = \sum_{k=1}^{\infty} \sum_{n=k}^{\infty} ([z^{n-k}] G(z)\mathcal{W}(z)^{k-1}) x^n.
$$
The inner sum over $n$ is a convolution that results in the generating function $G(x)\mathcal{W}(x)^{k-1}$ itself, multiplied by $x^k$, i.e.
$$
\sum_{n=k}^{\infty} ([z^{n-k}] G(z)\mathcal{W}(z)^{k-1}) x^n = x^k G(x)\mathcal{W}(x)^{k-1}.
$$
Thus, $\mathcal{B}(x)$ can therefore be written as the geometric series
\begin{align*}
\mathcal{B}(x) &= \sum_{k=1}^{\infty} x^k G(x)\mathcal{W}(x)^{k-1} = xG(x) \sum_{k=1}^{\infty} (x\mathcal{W}(x))^{k-1} = \frac{xG(x)}{1-x\mathcal{W}(x)}.
\end{align*}

First, we compute $G(x) = \mathcal{W}(x) + D(x)$:
\begin{align*}
G(x) & = \frac{1-2x+2x^2}{(1-2x)^2}.
\end{align*}
Next, we compute the denominator of $\mathcal{B}(x)$ as
\begin{align*}
1 - x\mathcal{W}(x) &= \frac{1-3x+x^2}{1-2x}.
\end{align*}
Finally, substituting these into the expression for $\mathcal{B}(x)$ yields
\begin{align*}
\mathcal{B}(x) &= \frac{x \cdot \frac{1-2x+2x^2}{(1-2x)^2}}{\frac{1-3x+x^2}{1-2x}} = x \cdot \frac{1-2x+2x^2}{(1-2x)^2} \cdot \frac{1-2x}{1-3x+x^2} = \frac{x(1-2x+2x^2)}{(1-2x)(1-3x+x^2)}.
\end{align*}
The ordinary generating function for the sequence $b_n$ thus agrees with the generating function $\mathcal{A}^{(312, 4321)}(x)$ from Lemma \ref{lem: affine for 4321}, as desired.
\end{proof}

\subsubsection{Direct approach}\label{subsubsec: direct approach}
We now handle the remaining cases in Theorem \ref{thm: TITO enumeration length 4} by direct computation.

\begin{prop}
    The number of $(312, 1234)$-avoiding TITOs is $f_n^{(312, 1234)}= 1.$ Specifically, the only $(312, 1234)$-avoiding $n$-TITO is $[\underline{n, \dots, 2,1}]$.
\end{prop}
\begin{proof}
    Let $\preceq$ be a $(312, 1234)$-avoiding $n$-TITO. Since $\preceq$ avoids $1234$, it cannot have a waxing block. By Lemma \ref{lem:312 blocks}, $\preceq$ must have exactly one waning block. Furthermore, the elements of the waning block must appear in decreasing order, so we must have $\preceq \ = [\underline{n, \dots, 1}]$, which clearly avoids 1234 and 312. Thus, $f_n^{(312, 1234)}= 1$.
\end{proof}

\begin{prop}
    The number of $(312, 1243)$-avoiding TITOs is $f_n^{(312, 1243)}= 2.$
\end{prop}
\begin{proof}
    Since $\preceq$ avoids $1243$, it cannot consist of a waxing block followed by a waning block, so by Lemma \ref{lem:312 blocks}, $\preceq$ must have either exactly one waxing block or exactly one waning block.

    If $\preceq$ consists of exactly one waxing block, then it is also an affine permutation. From Theorem \ref{thm: 312 p avoiding affine}, we know $a_n^{(312, 1243)} = 1$. If $\preceq$ consists of exactly one waning block, then the elements of the waning block appear in decreasing order, so we must have $\preceq \ = [\underline{n, \dots, 1}]$, which clearly avoids 1243 and 312. Thus, $f_n^{(312, 1243)}= 2$.
\end{proof}

\begin{prop}
    The number of $(312, 4321)$-avoiding TITOs is 
    \begin{align*}
        f_n^{(312, 4321)}= a_n^{(312, 4321)} = \left(\frac{3+\sqrt{5}}{2}\right)^n + \left(\frac{3-\sqrt{5}}{2}\right)^n - 2^n.
    \end{align*}
\end{prop}
\begin{proof}
    Any 4321-avoiding TITO cannot have a waning block. Thus, by Lemma \ref{lem:312 blocks}, any $(312, 4321)$-avoiding TITO must also be an affine permutation, so $f_n^{(312, 4321)} = a_n^{(312, 4321)}$.
\end{proof}

\begin{prop}
    The number of $(312, 2341)$-avoiding TITOs is 
    \begin{align*}
        f_n^{(312, 2341)}= a_n^{(312, 2341)} + 1 = -2 + \lambda_1^{n}+\lambda_2^{n}+\lambda_3^{n}.
    \end{align*}
\end{prop}
\begin{proof}
    Any 2341-avoiding TITO cannot have a waxing block followed by a waning block. Thus, by Lemma \ref{lem:312 blocks}, every $(312, 2341)$-avoiding TITO other than $[\underline{n, \dots, 1}]$ must be a $(312, 2341)$-avoiding affine permutation. Thus, $f_n^{(312, 2341)} = a_n^{(312, 2341)} + 1$.
\end{proof}

\begin{prop}
    For $p \in \{3421, 2431\}$, the number of $(312, p)$-avoiding TITOs is 
    \begin{align*}
        f_n^{(312, 3421)}= f_n^{(312, 2431)} = a_n^{(312, 4321)} + 1 = \left(\frac{3+\sqrt{5}}{2}\right)^n + \left(\frac{3-\sqrt{5}}{2}\right)^n - 2^n + 1.
    \end{align*}
\end{prop}
\begin{proof}
    Any 3421-avoiding TITO cannot have a waxing block followed by a waning block. Thus, by Lemma \ref{lem:312 blocks}, every $(312, 3421)$-avoiding TITO other than $[\underline{n, \dots, 1}]$ must be a $(312, 3421)$-avoiding affine permutation. Since $[\underline{n, \dots, 1}]$ is $(312, 3421)$-avoiding, we get \allowbreak $f_n^{(312, 3421)} \allowbreak = a_n^{(312, 3421)} + 1 = a_n^{(312, 4321)} + 1$ by Theorem \ref{thm: 312 p avoiding affine}. 

    The argument for $f_n^{(312, 2431)}$ is identical.
\end{proof}

\begin{prop}   
    The number of $(312, 3241)$-avoiding TITOs is
    $$f_n^{(312, 3241)} = f_n^{(312, 4321)} + 2^n - 1 = \left(\frac{3+\sqrt{5}}{2}\right)^n + \left(\frac{3-\sqrt{5}}{2}\right)^n - 1.$$
\end{prop}

\begin{proof}
    From Theorem \ref{thm: 312 p avoiding affine}, we have that $$a_n^{(312, 3241)} = a_n^{(312, 4321)} = \left(\frac{3+\sqrt{5}}{2}\right)^n + \left(\frac{3-\sqrt{5}}{2}\right)^n - 2^n.$$ It therefore remains to show that $b_n^{(312, 3241)} = 2^n - 1$.

    Let $\preceq$ be a projective $312$-avoiding TITO. By Lemma \ref{lem:312 blocks}, $\preceq$ either is the single waning block $[\underline{n, \dots, 1}]$ or has a waxing block followed by a decreasing waning block. In the latter case, we claim $\preceq$ is also $3241$-avoiding if and only if the waxing block is also increasing, that is, the waxing block contains no inversion.

    The backward direction is immediate. To show the forward direction, let us assume $(a, b)$ is an inversion within the waxing block. Since the waning block is decreasing, there exist $c, d$ in the waning block such that $b \preceq a \preceq c\preceq d$ with $d < a < b < c$, which forms a $3241$ pattern.
    
    The TITO $[\underline{n, \dots, 1}]$ is also $(312, 3241)$-avoiding. This is the only remaining case because each residue class $i \in [n]$ may belong to either the waxing block or the waning block, but we cannot have all elements in the waxing block. Thus, we get $b_n^{(312, 3241)} = 2^n - 1$.
\end{proof}

\begin{lem}\label{lem:increasing waxing}
    Let $\preceq$ be a $312$-avoiding TITO. For $p \in \{2134, 2143, 1324\}$, the TITO $\preceq$ is also $p$-avoiding if and only if elements of the waxing block appear in increasing order and elements of the waning block appear in decreasing order.
\end{lem}
\begin{proof}
    We first prove the reverse implication. If $\preceq$ consists of only one block, then its elements appear in increasing order when the block is waxing and in decreasing order when it is waning, by hypothesis. In either case, $\preceq$ avoids $p$.

    If $\preceq$ consists of a waxing block that occurs in increasing order followed by a waning block that occurs in decreasing order, we see that $\preceq$ also avoids $p$ and $312$.

    Now, we prove the forward implication via the contrapositive. If $\preceq$ has a waning block that does not occur in strictly decreasing order, then by Lemma \ref{lem:312 blocks}, $\preceq$ is not $312$-avoiding. If $\preceq$ has a waxing block that does not appear in increasing order, then some window of this block contains a lower wall $(a,b)$. By Lemma~\ref{lem:windows of 312}, we have $a<b<a+n$, so
    \begin{align*}
        b \preceq a \preceq a+n \preceq a+2n \quad \text{is a $2134$ pattern}, \\
        b \preceq a \preceq b+n \preceq a+n \quad \text{is a $2143$ pattern}, \\
        a-n \preceq b \preceq a \preceq b+n \quad \text{is a $1324$ pattern}.
    \end{align*}
    The lemma follows.
\end{proof}

\begin{prop}
For $p \in \{2134, 2143, 1324\}$, the number of $(312, p)$-avoiding TITOs is
    \begin{align*}
        f_n^{(312, p)} = 2^n.
    \end{align*}
\end{prop}
\begin{proof}
    By Lemma \ref{lem:increasing waxing}, we get
    \begin{align*}
        f_n^{(312, p)} = \binom{n}{0} + \binom{n}{1} + \cdots + \binom{n}{n} = (1+1)^n = 2^n,
    \end{align*}
    for $p \in \{2134, 2143, 1324\}$. 
\end{proof}

\begin{prop}
    For $p \in \{1432, 1342\}$, the number of $(312, p)$-avoiding TITOs is 
    \begin{align*}
        f_n^{(312, 1432)}= f_n^{(312, 1342)} = 2^n.
    \end{align*}
\end{prop}
\begin{proof}
    We first compute $f_n^{(312, 1432)}$. Since $\preceq$ avoids $1432$, it cannot consist of a waxing block followed by a waning block, so it must either have exactly one waxing block or exactly one waning block by Lemma \ref{lem:312 blocks}. 

    From Theorem \ref{thm: 312 p avoiding affine}, we know that $a_n^{(312, 1432)} = 2^n - 1$. Moreover, if $\preceq$ consists of exactly one waning block, the elements of the waning block occur in decreasing order, so we must have $\preceq \ = [\underline{n, \dots, 1}]$, which clearly avoids 1432 and 312. Thus, we get $f_n^{(312, 1432)}= 2^n$.

    The argument for $f_n^{(312, 1342)}$ is identical. 
\end{proof}

\section{Future directions}\label{sec: future directions}
Pattern avoidance in TITOs is a new topic at the intersection of pattern avoidance and algebraic combinatorics. As such, there are several promising avenues for future research, many of which we list below.

\begin{enumerate}
    \item One natural direction is to continue the investigation of pattern avoidance for longer patterns. For instance, one could pursue the systematic enumeration of TITOs that simultaneously avoid two patterns of length $4$.
    
    \item A second direction is to refine and extend our enumerative results by introducing a second statistic, such as the inversion number, to find $q$-analogues for our results. For example, Biagioli, Jouhet, and Nadeau \cite{BJN19} provide a formula for enumerating $321$-avoiding affine permutations with respect to the inversion number using heaps of pieces. Extending this type of analysis from affine permutations to the broader class of TITOs would allow for a more granular combinatorial understanding of these orders.
    
    \item The TITOs in this paper are connected to the affine symmetric group $\tilde{S}_{n}$, which is a Coxeter group of type~$\tilde{A}_{n-1}$. There are analogous objects for other affine Coxeter groups, such as types~$\tilde{B}_r, \tilde{C}_r,$ or~$\tilde{D}_r$. For instance, for a positive odd integer $N = 2r+1$, a type~$\tilde{C}_r$ TITO is a total order~$\preceq$ on the integers such that $i \preceq j$ if and only if $i + N \preceq j + N$ and $-j \preceq -i$ (see \cite{BS24}). One compelling open problem is to determine whether our enumerative techniques, such as the bijection with arc diagrams or the decomposition into waxing and waning blocks, can be adapted to these other affine settings.
\end{enumerate}

\section*{Acknowledgments}
This research was conducted at the University of Minnesota Duluth REU with support from Jane Street Capital, NSF Grant $2409861$, and donations from Ray Sidney and Eric Wepsic. I would like to thank Colin Defant for his invaluable guidance throughout the research process. I would also like to extend special thanks to Eliot Hodges, Rupert Li, Ilaria Seidel, and Isaac Rajagopal for their helpful suggestions during the editing process. Finally, I am very grateful to Joe Gallian and Colin Defant for organizing the Duluth REU and for the invitation to participate.

\end{document}